\documentclass [twoside,reqno, 12pt] {amsart}

\usepackage{tikz}

\usepackage{amsfonts}
\usepackage{amssymb}
\usepackage{a4}
\usepackage{color}

\newtheorem{thm}{Theorem}[section]

\newtheorem{lem}[thm]{Lemma}
\newtheorem{prop}[thm]{Proposition}

\newtheorem{rem}[thm]{Remark}

\theoremstyle{definition}

\numberwithin{equation}{section}

\renewcommand{\Re}{\hbox{Re}\,}
\renewcommand{\Im}{\hbox{Im}\,}

\newcommand{\C}{\mathbb{C}}

\renewcommand{\div}{\operatorname{div}}

\newcommand{\R}{\mathbb{R}}

\newcommand{\supp}{\operatorname{supp}}

\newcommand{\Z}{\mathbb{Z}}

\def\hat{\widehat}
\def\tilde{\widetilde}
\def \bfo {\begin {eqnarray*} }
\def \efo {\end {eqnarray*} }
\def \ba {\begin {eqnarray*} }
\def \ea {\end {eqnarray*} }
\def \beq {\begin {eqnarray}}
\def \eeq {\end {eqnarray}}
\def \supp {\hbox{supp }}

\def \p {\partial}

\def\hat{\widehat}
\def\tilde{\widetilde}
\def \bfo {\begin {eqnarray*} }
\def \efo {\end {eqnarray*} }
\def \ba {\begin {eqnarray*} }
\def \ea {\end {eqnarray*} }
\def \beq {\begin {eqnarray}}
\def \eeq {\end {eqnarray}}
\def \supp {\hbox{supp }}

\def \p {\partial}

\begin{document}

 \title[Calder\'on problem with partial data]{The Calder\'on problem with partial data for conductivities with $3/2$ derivatives}

\author[Krupchyk]{Katya Krupchyk}

\address
        {K. Krupchyk, Department of Mathematics\\
University of California, Irvine\\ 
CA 92697-3875, USA }

\email{katya.krupchyk@uci.edu}

\author[Uhlmann]{Gunther Uhlmann}

\address
       {G. Uhlmann, Department of Mathematics\\
       University of Washington\\
       Seattle, WA  98195-4350\\
       USA\\
       Department of Mathematics and Statistics\\ 
       University of Helsinki\\ 
       Finland\\
        and Institute for Advanced Study of the Hong Kong University of Science and Technology}
\email{gunther@math.washington.edu}

\maketitle

\begin{abstract}

We extend a global uniqueness result for the Calder\'on problem with partial data, due to Kenig--Sj\"ostrand--Uhlmann \cite{KenSjUhl2007},  to the case of less regular conductivities. Specifically, we show that in dimensions $n\ge 3$, the knowledge of the Diricihlet--to--Neumann map, measured on possibly very small subsets of the boundary, determines uniquely a conductivity having essentially $3/2$ derivatives in an $L^2$ sense.  

\end{abstract}

\section{Introduction}

Let $\Omega\subset\R^n$, $n\ge 3$, be a bounded open set with $C^2$ boundary,  and let $\gamma\in W^{1,\infty}(\Omega)$ be a real-valued function such that $\gamma>0$ on $\overline{\Omega}$, representing the conductivity of the domain $\Omega$.  Given a voltage potential $f\in H^{\frac{1}{2}}(\p \Omega)$ on the boundary of $\Omega$, the conductivity equation for the electric potential $u\in H^{1}(\Omega)$ in $\Omega$, under the assumption of no sources or sinks of currents, is given by
\begin{equation}
\label{eq_int_conductivity}
\begin{aligned}
L_\gamma u&=\hbox{div} (\gamma \nabla u)=0 \quad \text{in}\quad \Omega,\\  
u|_{\partial \Omega}&=f. 
\end{aligned}
\end{equation}
Associated to the problem \eqref{eq_int_conductivity} is the Dirichlet--to--Neumann map 
\[
\Lambda_\gamma: H^{\frac{1}{2}}(\p \Omega)\to H^{-\frac{1}{2}}(\p \Omega), \quad 
\Lambda_\gamma (f)= \gamma {\partial_\nu u }|_{\partial\Omega},
\]
where $\nu$ is the unit outer normal to the boundary of $\Omega$. The  Dirichlet--to--Neumann map $\Lambda_\gamma$ encodes the voltage to current measurements performed along the boundary of $\Omega$.  

The inverse conductivity problem, posed by Calder\'on  in \cite{Calderon}, studies the question whether the Dirichlet--to--Neumann map $\Lambda_\gamma$, given on the boundary of $\Omega$,  determines the conductivity $\gamma$ inside of   $\Omega$.    This problem is of significance in geophysical prospection, and it  has more recently been proposed as a possible diagnostic tool in medical imaging.  We refer to \cite{Uhlmann_seeing} for a recent comprehensive survey of the work on this problem.

In dimensions $n\ge 3$,  the first global uniqueness result for the inverse conductivity problem was established in \cite{Kohn_Vogelius_1984} for real-analytic conductivities. This was followed by \cite{Syl_Uhl_1987}, proving that  if the conductivities $0 <\gamma_1,\gamma_2\in C^2(\overline{\Omega})$  are such that $\Lambda_{\gamma_1}=\Lambda_{\gamma_2}$, then $\gamma_1=\gamma_2$ in $\Omega$. Subsequently, the regularity of the conductivity was relaxed to $\frac{3}{2}+\delta$ derivatives, $\delta>0$, on the scale of H\"older spaces, in \cite{Brown_1996}. The global uniqueness was further obtained for $W^{\frac{3}{2},\infty}$ conductivities in \cite{PPU_2003} and for conductivities in $W^{\frac{3}{2},p}$, with $p>2n$, in \cite{Brown_Torres_2003}. The recent breakthrough paper \cite{Hab_Tataru} established the global uniqueness for $C^1$ conductivities and Lipschitz continuous conductivities close to the identity.  The latter smallness condition was removed in \cite{Caro_Rogers}, thereby proving a long standing conjecture in the field. The global uniqueness for bounded conductivities in $W^{1,n}$, with $n=3,4$ was obtained in \cite{Haberman}.

Much less is known if the Dirichlet--to--Neumann map $\Lambda_\gamma$ is measured only on a portion of the boundary. The first result in this direction is due to \cite{Bukhgeim_Uhlmann_2002}, proving that if we measure the Dirichlet--to--Neumann map restricted to, roughly speaking, slightly more than half of the boundary, then we can determine a $C^2(\overline{\Omega})$ conductivity in $\Omega$ uniquely. The main technical tool in \cite{Bukhgeim_Uhlmann_2002} is boundary Carleman estimates with linear weights. The result of \cite{Bukhgeim_Uhlmann_2002} has been improved significantly in \cite{KenSjUhl2007}, still for $C^2(\overline{\Omega})$ conductivities, by showing that measuring the Dirichlet--to--Neumann map on a possibly very small open subset of the boundary, with the precise shape depending on the geometry of the domain,  we can determine the conductivity uniquely. Here rather than working with linear weights, a broader class of limiting Carleman weights was introduced and employed. 

Another approach to the partial data inverse problems is due to \cite{Isakov_2007}, and it is based on reflection arguments. In this approach, the subset of the boundary, where the measurements are performed is such that the inaccessible part of the boundary is a subset of a hyperplane or a sphere.  The article \cite{Kenig_Salo_2013} unifies and extends  the approaches of \cite{Bukhgeim_Uhlmann_2002},  \cite{KenSjUhl2007},  and \cite{Isakov_2007}. The linearized  Calder\'on  problem with partial data is studied in \cite{DKSU_2009} and \cite{Sjostrand_Uhlmann_analytic}. We refer to \cite{Kenig_Salo_servey} for a survey on the Calder\'on problem with partial data. 

Of great significance is the issue of reducing the regularity of the conductivity in the Calder\'on problem with partial data. In this direction, the result of \cite{Bukhgeim_Uhlmann_2002} was extended to conductivities of class $W^{\frac{3}{2}+\delta, 2n}(\Omega)$, $\delta>0$, in  \cite{Knudsen_2006}, and to  conductivities of class $C^1(\overline{\Omega})\cap H^{\frac{3}{2}}(\Omega)$ in  \cite{Zhang_Guo_2012}.  The recent paper \cite{Rodriguez} extended the partial data result of \cite{Bukhgeim_Uhlmann_2002} to the more general geometric setting by considering the Calder\'on problem on an admissible Riemannian manifold, assuming that the conductivity is of class $W^{\frac{3}{2}+\delta, 2n}$, $\delta>0$, as in \cite{Knudsen_2006}.  We refer to \cite{DKSaloU_2009} and \cite{Kenig_Salo_2013} for the study of the  Calder\'on problem in this geometric setting. 

Using a link between partial data results of type \cite{Bukhgeim_Uhlmann_2002} on an admissible Riemannian manifold and partial data results of type  \cite{KenSjUhl2007} on $\R^n$, the paper \cite{Rodriguez} relaxes the regularity of the conductivity in the partial data result of \cite{KenSjUhl2007} to $W^{\frac{3}{2}+\delta, 2n}(\Omega)$, $\delta>0$. Let us mention that the proof in \cite{Rodriguez} relies on boundary Carleman estimates with linear weights on admissible manifolds and the invertibility of the attenuated ray transform on simple manifolds.   

 In the present article, we shall further relax the regularity assumptions on the conductivity in   the partial data result of \cite{KenSjUhl2007}. Specifically, we are able to treat conductivities of class $C^{1,\delta}(\overline{\Omega})\cap H^{\frac{3}{2}}(\Omega)$  and conductivities in $W^{1,\infty}(\Omega)\cap H^{\frac{3}{2}+\delta}(\Omega)$. Here $0<\delta<1/2$ is arbitrarily  small but fixed.  When doing so, unlike \cite{Rodriguez}, we work with the conductivity equation directly in the Euclidean setting, and thus, following \cite{KenSjUhl2007}, we consider general limiting Carleman weights and establish boundary Carleman estimates in this context.

Let us now proceed to describe the precise assumptions and results.  First recall the definition of some standard function spaces needed in this paper. The Sobolev space $W^{s,p}(\R^n)$, with $s\in \R$ and $1< p<\infty$, is defined as follows, 
\[
W^{s,p}(\R^n)=\{u\in\mathcal{S}'(\R^n): \mathcal{F}^{-1}((1+|\xi|^2)^{s/2}\hat u)\in L^p(\R^n)\},
\]
where $\hat u$ is the Fourier transform of $u$, and $\mathcal{F}^{-1}$ is the inverse Fourier transform. For $s\ge 0$, we define the space $W^{s,p}(\Omega)$ as the image of the space $W^{s,p}(\R^n)$ under the map $u\mapsto u|_{\Omega}$.  When $p=2$, we shall write $H^s(\R^n)=W^{s,2}(\R^n)$ and $H^s(\Omega)=W^{s,2}(\Omega)$. Let  $C^{0,\delta}(\overline{\Omega})$, $0<\delta\le 1$, be the space of H\"older continuous functions on $\overline{\Omega}$, and let 
\[
C^{1,\delta}(\overline{\Omega})=\{u\in C^1(\overline{\Omega}): \nabla u\in C^{0,\delta}(\overline{\Omega})\}.
\]
Finally, recall the space 
\[
W^{1,\infty}(\Omega)=\{u\in L^\infty(\Omega): \nabla u\in L^\infty(\Omega)\}, 
\]
which can be identified with the space $C^{0,1}(\overline{\Omega})$ of Lipschitz continuous functions on $\overline{\Omega}$.

Let $x_0\in \R^n\setminus\overline{\text{ch}(\Omega)}$, where $\text{ch}(\Omega)$ is the convex hull of $\Omega$. Following \cite{KenSjUhl2007}, we define the front face of $\p \Omega$ with respect to $x_0$ by 
\begin{equation}
\label{eq_int_front}
F(x_0)=\{x\in \p \Omega: (x-x_0)\cdot \nu(x)\le 0\},
\end{equation}
and let $\tilde F$ be an open neighborhood of $F(x_0)$ in $\p \Omega$.

The main result of this paper is as follows. 
\begin{thm}
\label{thm_main}
Let $\Omega\subset\R^n$, $n\ge 3$, be a bounded open set with $C^2$ boundary, and let $\gamma_1,\gamma_2$ be such that either 
\begin{itemize}

\item[(i)] $\gamma_1,\gamma_2\in C^{1,\delta}(\overline{\Omega})\cap H^{\frac{3}{2}}(\Omega)$,

\end{itemize}
or 
\begin{itemize}

\item[(ii)]$\gamma_1,\gamma_2 \in W^{1,\infty}(\Omega)\cap H^{\frac{3}{2}+\delta}(\Omega)$, 
\end{itemize}
where $0<\delta<1/2$  is arbitrarily small but fixed.

Assume that $\gamma_1,\gamma_2>0$ in $\overline{\Omega}$,  $\gamma_1=\gamma_2$ on $\p \Omega\setminus \tilde F$, and that in the case \emph{(i)}  
$\p_\nu\gamma_1=\p_\nu\gamma_2$ on $\p \Omega\setminus \tilde F$, and in the case \emph{(ii)} $\p_\nu\gamma_1=\p_\nu\gamma_2$ in $H^{\delta}(\p \Omega)$. Assume furthermore that 
\[
\Lambda_{\gamma_1} f|_{\tilde F}=\Lambda_{\gamma_2} f_{\tilde F}\quad \text{for all}\quad f\in H^{\frac{1}{2}}(\p \Omega). 
\]
Then $\gamma_1=\gamma_2$ in $\Omega$. 
 
\end{thm}

\textbf{Remark 1}.  As observed in \cite{KenSjUhl2007}, if $\Omega$ is strictly convex, then the set $F(x_0)$ can be made arbitrarily small, by choosing the point $x_0$ suitably.

\textbf{Remark 2}. Theorem \ref{thm_main} is applicable to conductivities in the Sobolev spaces $W^{\frac{3}{2}+\delta, 2n}(\Omega)$ and $W^{\frac{3}{2},2n+\delta}(\Omega)$, considered in partial data results of  \cite{Knudsen_2006} and \cite{Rodriguez}, and in the full data result of \cite{Brown_Torres_2003}, respectively.  Indeed,  by Sobolev embedding, we have 
\[
W^{\frac{3}{2}+\delta, 2n}(\Omega) \subset C^{1,\delta}(\overline{\Omega}), \quad W^{\frac{3}{2},2n+\delta}(\Omega)\subset C^{1,\frac{\delta}{4n+2\delta}}(\overline{\Omega}),
\]
see \cite[Theorem 7.63]{Adams_book}. It is also easy to see that 
\[
W^{\frac{3}{2}+\delta, 2n}(\Omega) \subset H^{\frac{3}{2}+\delta}(\Omega), \quad
W^{\frac{3}{2},2n+\delta}(\Omega)\subset H^{\frac{3}{2}}(\Omega).
\] 

\textbf{Remark 3}. The existing proofs of the global uniqueness results in the Calder\'on problem for conductivities with fewer than $3/2$ derivatives, in the case of the full data, developed in \cite{Hab_Tataru},  \cite{Haberman} and \cite{Caro_Rogers}, rely crucially on the linear nature of the limiting Carleman weights involved and make use of some averaging techniques.  On the other hand, a key point in the partial data result of \cite{KenSjUhl2007} is to use more general non-linear limiting Carleman weights. Therefore, to go below $3/2$ derivatives in the partial data result of \cite{KenSjUhl2007}, it seems that a new approach would be needed.

Let us now describe the main ideas in the proof of Theorem \ref{thm_main}. 
A fundamental approach to the inverse conductivity problem, which we shall also follow in this work, is based on construction of the so called complex geometric optics solutions for the conductivity equation, see \cite{Syl_Uhl_1987}, \cite{KenSjUhl2007}. To this end, using the identify, 
\[
\gamma^{-1/2}\circ L_\gamma\circ \gamma^{-1/2}=\Delta-q, \quad
q=\frac{\Delta\gamma^{1/2}}{\gamma^{1/2}},
\]
we may reduce the problem of construction of such solutions to the corresponding problem for the Schr\"odinger equation $(-\Delta+q)v=0$ in $\Omega$. Here the potential $q\in L^\infty(\Omega)$ provided that $\gamma\in C^2(\overline{\Omega})$, while if $\gamma$ is merely Lipschitz continuous, the corresponding potential $q$ becomes a distribution in $H^{-1}(\Omega)$.  In Subsection \ref{sec_CGO_Lip}, using this reduction, we construct complex geometric optics solutions with limiting Carleman weights for the conductivity equation in the case of conductivities
of class $W^{1,\infty}(\Omega)$. Unfortunately, it turns out that the remainder estimates for such solutions are not strong enough to solve the inverse problem in this case, even for the full data. In Subsection \ref{sec_CGO_H_3_2} we therefore sharpen the remainder estimates for conductivities of class $W^{1,\infty}(\Omega)\cap H^{\frac{3}{2}}(\Omega)$. It turns out that these sharpened estimates do suffice to control the interior terms in some crucial integral identity, used to establish the equality of the conductivities, see Section \ref{sec_integral identity} and  \cite{Brown_1996}.

Another crucial ingredient needed to establish global uniqueness in the  Calder\'on problem with partial data is a Carleman estimate with boundary terms, see \cite{Bukhgeim_Uhlmann_2002},  \cite{KenSjUhl2007}, and \cite{DKSU_2007}. Since our conductivities give rise to potentials which are singular, when deriving such estimates, it turns out to be more convenient to work directly with the conductivity equation, which we write in the form 
\begin{equation}
\label{eq_cont_rewritten}
-\Delta u - A\cdot\nabla u=0 \quad \text{in}\quad \Omega,
\end{equation}  
where $A=\nabla \log \gamma\in L^\infty$. Boundary Carleman estimates with limiting Carleman weights for first order perturbations of the Laplacian have been established in \cite{DKSU_2007}. However, it seems that their direct application does not allow one to get rid of some boundary terms, computed over the inaccessible portion of the boundary.  Indeed, applying the boundary Carleman estimate of \cite{DKSU_2007} will produce a term of magnitude $\mathcal{O}(h^{-1/2})\|\nabla\log \gamma_1-\nabla \log\gamma_2\|_{L^2(\Omega)}$, $0 < h \ll 1$, which cannot be controlled as $h\to 0$. 

To overcome this difficulty, we shall follow an idea of  \cite{PPU_2003}, \cite{Knudsen_2006}, which consists of replacing the conductivity equation \eqref{eq_cont_rewritten} by its conjugated version which is of the form,
\begin{equation}
\label{eq_cont_rewritten_2}
-\Delta u+(A_h- A) \cdot \nabla u+V_h u=0 \quad \text{in}\quad \Omega.
\end{equation}
Here $A_h$ is a regularization of  $A$ and $V_h$ is a suitable potential.  An advantage of working with \eqref{eq_cont_rewritten_2} is that for $\gamma\in W^{1,\infty}(\Omega)\cap H^{3/2}(\Omega)$, we have $\|A_h- A\|_{L^2}=o(h^{1/2})$, as $h\to 0$. The price that we have to pay to work with 
\eqref{eq_cont_rewritten_2} is that we need to extend the boundary Carleman estimate of \cite{DKSU_2007} to the case of functions which need not vanish along the boundary of 
$\Omega$. This extension is carried out in Section \ref{sec_BCE}, and we hope that it might be of some independent interest. 

Let us finally remark that to get rid of the boundary terms in the integral identity of Section \ref{sec_integral identity}, we shall need a bit more regularity for the conductivities than $W^{1,\infty}(\Omega)\cap H^{3/2}(\Omega)$, as stated in Theorem \ref{thm_main}. Another technical reason for this additional regularity is that in the course of the proof, we need to extend the conductivities $\gamma_1$ and $\gamma_2$  to all of $\R^n$ so that  $\gamma_1=\gamma_2$ on $\R^n\setminus\overline{\Omega}$, and their regularity is preserved. 

The paper is organized as follows. In Section \ref{sec_CGO} we construct complex geometric optics solutions to the conductivity equation. Boundary Carleman estimates are established in Section \ref{sec_BCE}, and following \cite{Brown_1996}, we recall a basic integral identity in Section \ref{sec_integral identity}. Section \ref{sec_proof} is devoted to the proof of Theorem \ref{thm_main}. In Appendix \ref{app_estimates} we collect some standard approximation estimates needed in the main text, for the convenience of the reader. 

\vspace*{-5pt}
\section{Complex geometric optics solutions with limiting Carleman weights for conductivity equation} 

\label{sec_CGO}
\subsection{Lipschitz continuous conductivities}
\label{sec_CGO_Lip}
Let $\Omega\subset \R^n$, $n\ge 3$, be a bounded open set with $C^2$ boundary and let $\gamma\in W^{1,\infty}(\Omega)$ and $\gamma>0$ on $\overline{\Omega}$. We can extend $\gamma$ to a function on  $\R^n$ so that the extension, still denoted by $\gamma$, satisfies  
$0<\gamma\in W^{1,\infty}(\R^n)$ and $\gamma=1$ near infinity. 
Let 
\[
q=\frac{\Delta\gamma^{1/2}}{\gamma^{1/2}}=- \nabla \gamma^{1/2} \cdot \nabla\gamma^{-1/2} +\frac{1}{2}\Delta\log \gamma\in (H^{-1}\cap \mathcal{E}')(\R^n).
\]
Following \cite{Brown_1996},  \cite{Hab_Tataru}, we define the "multiplication by $q$" map 
\[
m_q:H^1(\R^n)\to H^{-1}(\R^n)
\]
by 
\begin{equation}
\label{eq_def_m_q}
\langle m_q(u),v \rangle_{\R^n}
=-\int_{\R^n} (\nabla \gamma^{1/2} \cdot \nabla\gamma^{-1/2})uvdx- \frac{1}{2}\int_{\R^n} \nabla \log \gamma \cdot \nabla(uv)dx,
\end{equation}
for $u,v\in H^1(\R^n)$. Here $\langle \cdot,\cdot\rangle_{\R^n}$ is the distribution duality on $\R^n$.  Whenever convenient we shall also view $m_q$ as a map $m_q: H^1(\Omega)\to H^{-1}(\Omega)$
given by
\begin{equation}
\label{eq_def_m_q_omega}
\langle m_q(u),v \rangle_{\Omega}
=-\int_{\Omega} (\nabla \gamma^{1/2} \cdot \nabla\gamma^{-1/2})uvdx- \frac{1}{2}\int_{\Omega} \nabla \log \gamma \cdot \nabla(uv)dx,
\end{equation}
for $u\in H^1(\Omega)$, $v\in C^\infty_0(\Omega)$. Here $\langle \cdot,\cdot\rangle_{\Omega}$ is the distribution duality on $\Omega$. 
Notice that when $u\in H^1(\R^n)\cap \mathcal{E}'(\overline{\Omega})$, the definitions \eqref{eq_def_m_q} and \eqref{eq_def_m_q_omega} agree on $\Omega$.

Following \cite{DKSU_2007, KenSjUhl2007}, we shall use the method of Carleman estimates to construct complex geometric optics solutions in $H^1(\Omega)$ for the Schr\"odinger equation 
\begin{equation}
\label{eq_2_Schr_eq}
-\Delta u +m_q(u)=0\quad \text{in}\quad\Omega. 
\end{equation}
We then know that $\gamma^{-1/2}u\in H^1(\Omega)$ satisfies  the conductivity equation 
\[
L_\gamma (\gamma^{-1/2}u)=0\quad \text{in}\quad \Omega.
\]

As in our work \cite{Krupchyk_Uhlmann_magnetic}, we shall rely on the Carleman estimate for the semiclassical Laplace operator $-h^2\Delta$ with a gain of two derivatives, established in   \cite{Salo_Tzou_2009}, see also \cite{KenSjUhl2007}.  Here $h>0$ is a small semiclassical parameter.  Let us proceed by recalling this estimate.  Let $\tilde \Omega$ be an open set in $\R^n$ such that $ \Omega\subset\subset\tilde \Omega$ and let
$\varphi\in C^\infty(\tilde \Omega,\R)$.  Consider the conjugated operator 
\[
P_\varphi=e^{\frac{\varphi}{h}}(-h^2\Delta) e^{-\frac{\varphi}{h}},
\]
with the semiclassical principal symbol
\begin{equation}
\label{eq_2_sem_principal}
p_\varphi(x,\xi)=\xi^2+2i\nabla \varphi\cdot \xi-|\nabla \varphi|^2, \quad x\in \tilde {\Omega},\quad  \xi\in \R^n. 
\end{equation}
We have for $(x,\xi)\in \overline{\Omega}\times \R^n$, $|\xi|\ge C\gg 1$, that $|p_\varphi(x,\xi)|\sim |\xi|^2$ so that $P_\varphi$ is elliptic at infinity, in the semiclassical sense. 
Following \cite{KenSjUhl2007}, we say that $\varphi$ is a limiting Carleman weight for $-h^2\Delta$ in $\tilde \Omega$, if $\nabla \varphi\ne 0$ in $\tilde \Omega$ and the Poisson bracket of $\Re p_\varphi$ and $\Im p_\varphi$ satisfies, 
\[
\{\Re p_\varphi,\Im p_\varphi\}(x,\xi)=0 \quad \textrm{when}\quad p_\varphi(x,\xi)=0, \quad (x,\xi)\in \tilde{\Omega}\times \R^n. 
\]
Examples of limiting Carleman weights  are  linear weights $\varphi(x)=\alpha\cdot x$, $\alpha\in \R^n$, $|\alpha|=1$, and logarithmic weights $\varphi(x)=\log|x-x_0|$,  with $x_0\not\in \tilde \Omega$.  In this paper we shall only use the logarithmic weights. 

Our starting point is the following result due to \cite{Salo_Tzou_2009}.
\begin{prop}
Let $\varphi$ be a limiting Carleman weight for the semiclassical Laplacian on $\tilde \Omega$, and let $\tilde \varphi=\varphi+\frac{h}{2\varepsilon}\varphi^2$.  Then 
for $0<h\ll \varepsilon\ll 1$ and $s\in\R$, we have
\begin{equation}
\label{eq_Carleman_lap}
\frac{h}{\sqrt{\varepsilon}}\|u\|_{H^{s+2}_{\emph{scl}}(\R^n)}\le C\|e^{\tilde \varphi/h}(-h^2\Delta)e^{-\tilde \varphi/h}u\|_{H^s_{\emph{scl}}(\R^n)},
\quad C>0,
\end{equation}
 for all $u\in C^\infty_0(\Omega)$.  
\end{prop}
Here 
\[
\|u\|_{H^s_{\textrm{scl}}(\R^n)}=\|\langle hD \rangle^s u\|_{L^2(\R^n)},\quad \langle\xi  \rangle=(1+|\xi|^2)^{1/2},
\]
is the natural semiclassical norm in the Sobolev space  $H^s(\R^n)$, $s\in\R$.

We have the following result. 
\begin{prop}
Let $\varphi\in C^\infty(\tilde\Omega,\R)$ be a limiting Carleman weight for the semiclassical Laplacian on $\tilde \Omega$, and let $0<\gamma\in W^{1,\infty}(\R^n)$ be such that $\gamma=1$ near infinity. Then 
for all $h>0$ sufficiently small, we have 
\begin{equation}
\label{eq_Carleman_schr}
h\|u\|_{H^{1}_{\emph{scl}}(\R^n)}\le C\|e^{\varphi/h}(-h^2\Delta +h^2m_q)e^{-\varphi/h}u\|_{H^{-1}_{\emph{scl}}(\R^n)},
\end{equation}
 for all $u\in C^\infty_0(\Omega)$.  
\end{prop}

\begin{proof}
In order to prove the estimate \eqref{eq_Carleman_schr} it will be convenient to use the following characterization of the semiclassical norm in the Sobolev space $H^{-1}(\R^n)$, 
\begin{equation}
\label{eq_charac_H_-}
\|u\|_{H_{\textrm{scl}}^{-1}(\R^n)}=\sup_{0\ne v\in C^\infty_0(\R^n)}\frac{|\langle u,v\rangle_{\R^n} |}{\|v\|_{H_{\textrm{scl}}^{1}(\R^n)}}.
\end{equation}

Let $\tilde \varphi=\varphi+\frac{h}{2\varepsilon}\varphi^2$ with $0<h\ll \varepsilon\ll 1$, and let $u \in C^\infty_0(\Omega)$.  Then for all $0\ne v\in C^\infty_0(\R^n)$, we have 
 \begin{align*}
|\langle  e^{\tilde \varphi/h} h^2m_q(&e^{-\tilde \varphi/h}u),v \rangle_{\R^n}|\\
&\le h^2 \int_{\R^n}|(\nabla \gamma^{1/2}\cdot \nabla \gamma^{-1/2})uv|dx+h^2\int_{\R^n}|\nabla \log\gamma\cdot \nabla(uv)|dx\\
&\le h^2\|\nabla \gamma^{1/2}\cdot \nabla \gamma^{-1/2}\|_{L^\infty(\R^n)}\|u\|_{L^2(\R^n)}\|v\|_{L^2(\R^n)}\\
&+2h \|\nabla \log \gamma\|_{L^\infty(\R^n)}\|u\|_{H^1_{\text{scl}}(\R^n)}\|v\|_{H^1_{\text{scl}}(\R^n)}\\
&\le \mathcal{O}(h)\|u\|_{H^1_{\text{scl}}(\R^n)}\|v\|_{H^1_{\text{scl}}(\R^n)},
\end{align*}
and therefore, uniformly in $\varepsilon$, 
\begin{equation}
\label{eq_2_3} 
\|e^{\tilde \varphi/h} h^2m_q(e^{-\tilde \varphi/h}u) \|_{H^{-1}_{\text{scl}}(\R^n)}\le \mathcal{O}(h)\|u\|_{H^1_{\text{scl}}(\R^n)}.
\end{equation}
Now choosing $\varepsilon >0$  sufficiently small but fixed, i.e. independent of $h$, we obtain from the estimate \eqref{eq_Carleman_lap} with $s=-1$ and the estimate \eqref{eq_2_3} that for all $h>0$ small enough, 
\[
\|e^{\tilde \varphi/h}(-h^2\Delta+ h^2m_q)(e^{-\tilde \varphi/h}u) \|_{H^{-1}_{\text{scl}}(\R^n)}\ge \frac{h}{C}\|u\|_{H^1_{\text{scl}}(\R^n)}, \quad C>0.
\]
This estimate together with the fact that  
\[
e^{-\tilde \varphi/h}u=e^{-\varphi/h}e^{-\varphi^2/(2\varepsilon)}u,
\]
implies \eqref{eq_Carleman_schr}. The proof is complete. 
\end{proof}

Now since the formal $L^2(\Omega)$ adjoint to the operator $e^{\varphi/h}(-h^2\Delta +h^2m_q)e^{-\varphi/h}$ is given by $e^{-\varphi/h}(-h^2\Delta +h^2m_q)e^{\varphi/h}$ and $-\varphi$ is also a limiting Carleman weight, by classical arguments involving the Hahn--Banach theorem, one converts the Carleman estimate  \eqref{eq_Carleman_schr} for the adjoint into the following solvability result, see \cite{Krupchyk_Uhlmann_magnetic} for the proof. 

\begin{prop}
\label{prop_solvability}
Let  $\gamma\in W^{1,\infty}(\Omega)$ be such that $\gamma>0$ on $\overline{\Omega}$,  and let $\varphi$ be a limiting Carleman weight for the semiclassical Laplacian on $\tilde \Omega$. If $h>0$ is small enough, then for any $v\in H^{-1}(\Omega)$, there is a solution $u\in H^1(\Omega)$ of the equation
\[
e^{\varphi/h}(-h^2\Delta +h^2m_q)e^{-\varphi/h}u=v\quad\textrm{in}\quad \Omega,
\]
which satisfies
\[
\|u\|_{H^1_{\emph{scl}}(\Omega)}\le \frac{C}{h}\|v\|_{H^{-1}_{\emph{scl}}(\Omega)}.
\]
\end{prop}
Here 
\begin{align*}
&\|u\|_{H^1_{\textrm{scl}}(\Omega)}^2=\|u\|_{L^2(\Omega)}^2+\|hDu\|_{L^2(\Omega)}^2,\\
&\|v\|_{H^{-1}_{\textrm{scl}}(\Omega)}=\sup_{0\ne \psi\in C_0^\infty(\Omega)}\frac{|\langle v,\psi\rangle_{\Omega}|}{\|\psi\|_{H^1_{\textrm{scl}}(\Omega)}}.
\end{align*}

Let us construct complex geometric optics solution to the Schr\"odinger equation \eqref{eq_2_Schr_eq}, i.e. solutions of the form,
\begin{equation}
\label{eq_2_CGO_main}
u(x;h)=e^{\frac{\varphi+i\psi}{h}}(a(x)+r(x;h)).
\end{equation}
Here $\varphi\in C^\infty(\tilde \Omega,\R)$ is a limiting Carleman weight for the semiclassical Laplacian on $\tilde \Omega$, $\psi\in C^\infty(\tilde \Omega,\R)$ is a solution to the eikonal equation $p_\varphi(x,\nabla \psi)=0$ in $\tilde \Omega$, where $p_\varphi$ is given by \eqref{eq_2_sem_principal}, i.e. 
\begin{equation}
\label{eq_2_eikonal}
|\nabla \psi|^2=|\nabla \varphi|^2,\quad \nabla \varphi\cdot \nabla\psi=0\quad \text{in}\quad \tilde\Omega,
\end{equation}
$a\in C^\infty(\overline{\Omega})$ is an amplitude, and $r$ is a correction term.

Following \cite{DKSU_2007, KenSjUhl2007}, we fix a point $x_0\in \R^n\setminus{\overline{\textrm{ch}(\Omega)}}$ and let the limiting Carleman weight  be 
\begin{equation}
\label{eq_varphi}
\varphi(x)=\frac{1}{2}\log|x-x_0|^2,
\end{equation}
and 
\begin{equation}
\label{eq_psi}
\psi(x)=\frac{\pi}{2}-\arctan\frac{\omega\cdot (x-x_0)}{\sqrt{(x-x_0)^2-(\omega\cdot(x-x_0))^2}}=\textrm{dist}_{\mathbb{S}^{n-1}}\bigg(\frac{x-x_0}{|x-x_0|},\omega\bigg),
\end{equation}
where $\omega\in \mathbb{S}^{n-1}$ is chosen so that $\psi$ is smooth near $\overline{\Omega}$. Thus, given  $\varphi$, the function  $\psi$ satisfies the eikonal equation \eqref{eq_2_eikonal} near $\overline{\Omega}$. 

Conjugating the operator $-h^2\Delta+h^2m_q$ by $e^{\frac{\varphi+i\psi}{h}}$ and using \eqref{eq_2_eikonal}, we get 
\begin{equation}
\label{eq_2_10_1}
e^{-\frac{(\varphi+i\psi)}{h}}\circ(-h^2\Delta+h^2m_q)\circ e^{\frac{\varphi+i\psi}{h}}=-h^2\Delta -2(\nabla \varphi+i\nabla \psi)\cdot h\nabla -h(\Delta \varphi+i\Delta\psi)+h^2m_q.
\end{equation}

Substituting \eqref{eq_2_CGO_main} into \eqref{eq_2_Schr_eq},  and using \eqref{eq_2_10_1}, we obtain the following equation for $r$,
\begin{equation}
\label{eq_2_10_2}
e^{-\frac{(\varphi+i\psi)}{h}}(-h^2\Delta+h^2m_q)(e^{\frac{\varphi+i\psi}{h}}r) =h^2\Delta a-h^2m_q(a),
\end{equation}
provided that $a$ satisfies the first transport equation,
\begin{equation}
\label{eq_2_10_3}
2(\nabla \varphi+i\nabla \psi)\cdot \nabla a +(\Delta \varphi+i\Delta\psi)a=0 \quad \text{in}\quad\Omega.
\end{equation}
Thanks to the works \cite{DKSU_2007}, \cite{KenSjUhl2007}, we know that the transport equation \eqref{eq_2_10_3} is of a Cauchy--Riemann type  and that it has a non-vanishing solution  $a\in C^\infty(\overline{\Omega})$.  

Applying now the solvability result of Proposition \ref{prop_solvability}, we conclude that for all $h>0$ small enough, there exists $r(x;h)\in H^1(\Omega)$ satisfying \eqref{eq_2_10_2}
such that 
\begin{equation}
\label{eq_2_10_4}
\|r\|_{H^1_{\text{scl}}(\Omega)}\le \mathcal{O}(h)+\mathcal{O}(h)\|m_q(a)\|_{H^{-1}_{\text{scl}}(\Omega)}.
\end{equation}

We shall next estimate $h\|m_q(a)\|_{H^{-1}_{\text{scl}}(\Omega)}$. Letting $0\ne v\in C^\infty_0(\Omega)$, we get 
\begin{equation}
\label{eq_2_10_5}
\begin{aligned}
|\langle  h m_q(a),v\rangle_{\Omega}|&\le h \int_{\Omega} |(\nabla \gamma^{1/2}\cdot \nabla \gamma^{-1/2}) a v|dx+h\bigg|\int_{\Omega} \nabla \log \gamma \cdot \nabla (av)dx\bigg|\\
&\le \mathcal{O}(h)\|v\|_{L^2(\Omega)}+\mathcal{O}(h) I,
\end{aligned}
\end{equation}
where 
\begin{equation}
\label{eq_2_10_5_1}
I:=\bigg|\int_{\Omega} a \nabla \log \gamma \cdot \nabla vdx\bigg|.
\end{equation}
Thus, we only need to estimate $hI$. 

We have $A:=\nabla \log \gamma \in (L^\infty\cap \mathcal{E}')(\R^n)\subset L^p(\R^n)$, $1\le p\le \infty$.  Let 
\begin{equation}
\label{eq_mollifier_def}
\Psi_\tau(x)=\tau^{-n}\Psi(x/\tau), \quad \tau>0,
\end{equation}
 be the usual mollifier with $\Psi\in C^\infty_0(\R^n)$, $0\le \Psi\le 1$, and 
$\int \Psi dx=1$.  Then $A_\tau=A*\Psi_\tau\in C_0^\infty(\R^n)$ and 
\begin{equation}
\label{eq_flat_est}
\|A-A_\tau\|_{L^2(\R^n)}=o(1), \quad \tau\to 0.
\end{equation}
An application of Young's inequality shows that 
\begin{equation}
\label{eq_flat_est_2}
 \|\p^\alpha A_\tau\|_{L^2(\R^n)}\le \|A\|_{L^2(\R^n)}\|\p^\alpha \Psi_\tau\|_{L^1(\R^n)}\le \mathcal{O}(\tau^{-|\alpha|}),\quad \tau\to 0, \quad |\alpha|\ge 0.
\end{equation}
Using \eqref{eq_flat_est} and \eqref{eq_flat_est_2},  we get integrating by parts, 
\begin{equation}
\label{eq_2_10_6}
\begin{aligned}
hI &\le \mathcal{O}(h)\int_\Omega|(A-A_\tau)\cdot \nabla v|dx+h\bigg|\int_\Omega a A_\tau\cdot \nabla v\bigg|\\
&\le \mathcal{O}(h)\|A-A_\tau\|_{L^2(\Omega)}\|\nabla v\|_{L^2(\Omega)}+h\| a\nabla A_\tau+A_\tau\cdot\nabla a \|_{L^2(\Omega)}\|v\|_{L^2(\Omega)}\\
&\le o_{\tau\to 0}(1)\|v\|_{H^1_{\text{scl}}(\Omega)}+h\mathcal{O}(\tau^{-1})\|v\|_{H^1_{\text{scl}}(\Omega)}.
\end{aligned}
\end{equation}
Choosing now $\tau=h^\sigma$ with some $0<\sigma<1$, we obtain from \eqref{eq_2_10_4}, \eqref{eq_2_10_5} and \eqref{eq_2_10_6} that $\|r\|_{H^1_{\text{scl}}(\Omega)}=o(1)$ as $h\to 0$. 

Summing up, we have the following result on the existence of complex geometric optics solutions for Lipschitz continuous conductivities. 

\begin{prop}
\label{prop_CGO_lipschitz}
Let $\Omega \subset\R^n$, $n\ge 3$, be  a bounded open set with $C^2$ boundary and let $\tilde \Omega\subset \R^n$ be an open set such that  $\Omega\subset\subset \tilde \Omega$ . Let  $\gamma\in W^{1,\infty}(\Omega)$ be such that $\gamma>0$ on $\overline{\Omega}$. 
 Then for all $h>0$ small enough, there exists a solution $u(x;h)\in H^1(\Omega)$ to the conductivity equation $L_\gamma u=0$ in $\Omega$, of the form
\begin{equation}
\label{eq_prop_CGO}
u(x;h)=\gamma^{-1/2}e^{\frac{\varphi+i\psi}{h}}(a(x)+r(x;h)),
\end{equation}
where $\varphi\in C^\infty(\tilde \Omega,\R)$ is a limiting Carleman weight for the semiclassical Laplacian on $\tilde \Omega$, $\psi\in C^\infty(\tilde \Omega,\R)$ is a solution to the eikonal equation \eqref{eq_2_eikonal}, $a\in C^\infty(\overline{\Omega})$ is a solution of the first transport equation \eqref{eq_2_10_3}, and the remainder term $r$ is such that $\|r\|_{H^1_{\emph{\text{scl}}}(\Omega)}=o(1)$ as $h\to 0$. 
\end{prop}

\subsection{Lipschitz continuous conductivities in  $H^{3/2}$}
\label{sec_CGO_H_3_2}
Let $\Omega\subset \R^n$, $n\ge 3$, be a bounded open set with $C^2$ boundary, and let $\gamma\in W^{1,\infty}(\Omega)\cap H^{\frac{3}{2}}(\Omega)$ and $\gamma>0$ on $\overline{\Omega}$.   In this subsection we shall improve the result of Proposition \ref{prop_CGO_lipschitz} by deriving sharpened estimates for the remainder $r$  in \eqref{eq_prop_CGO}.

Let us first show that we can extend $\gamma$ to a function $0<\gamma\in W^{1,\infty}(\R^n)$ such that $\gamma=1$ near infinity and $\gamma-1\in H^{3/2}(\R^n)$.  To this end, using a partition of unity, we see that it suffices to work locally near a point at $\p \Omega$, and flattening out $\p \Omega$ by a $C^2$ diffeomorphism, we may consider the problem of extending $\gamma\in W^{1,\infty}(\R^n_+)\cap H^{\frac{3}{2}}(\R^n_+)\cap\mathcal{E}'(\overline{\R^n_+})$ to all of $\R^n$. 

Following a standard argument, see \cite[Theorem 4.12]{Grubb_book}, we introduce the following linear operator on $C^\infty_{(0)}(\overline{\R^n_+})=\{u\in C^\infty(\overline{\R^n_+}):\supp (u) \text{ compact }\subset \overline{\R^n_+} \}$,
\[
(E u)(x',x_n)=\begin{cases} u(x',x_n), & x_n>0,\\
\sum_{j=1}^3 \lambda_j u(x',-jx_n)& x_n<0,
\end{cases}
\]
where $\lambda_j\in \R$ are determined by the system of equations,
\[
\sum_{j=1}^3 (-j)^k \lambda_j =1,\quad k=0, 1,2. 
\]
Then $E$ extends continuously to $E:L^2(\R^n_+)\to L^2(\R^n)$ and $E:H^2(\R^n_+)\to H^2(\R^n)$, and hence, by interpolation, 
\[
E:H^{s}(\R^n_+)\to H^{s}(\R^n), \quad 0\le s\le 2.
\]
One can easily check that  
\begin{align*}
E: C^1(\overline{\R^n_+})\cap H^{s}(\R^n_+)&\to 
C^1(\R^n)\cap H^{s}(\R^n), \quad 0\le s\le 2,\\
E: W^{1,\infty}(\R^n_+)\cap H^{s}(\R^n_+)&\to 
W^{1,\infty}(\R^n)\cap H^{s}(\R^n), \quad 0\le s\le 2,\\
E: C^{1,\delta}(\overline{\R^n_+})\cap H^{s}(\R^n_+)&\to 
C^{1,\delta}(\R^n)\cap H^{s}(\R^n), \quad 0\le s\le 2.
\end{align*}

Coming back to $\Omega$, we obtain that the conductivity $\gamma\in W^{1,\infty}(\Omega)\cap H^{\frac{3}{2}}(\Omega)$ such that  $\gamma>0$ on $\overline{\Omega}$ has an extension $\tilde \gamma\in W^{1,\infty}(\R^n)\cap H^{\frac{3}{2}}(\R^n)$ such that $\tilde \gamma>0$ in a neighborhood $V$ of $\overline{\Omega}$.  Letting $\varphi\in C^\infty_0(V)$ be such that $0\le \varphi\le 1$ and $\varphi=1$ near $\overline{\Omega}$, we see that  $\gamma=\tilde \gamma\varphi +1-\varphi$  satisfies the required properties.

Set $A=\nabla \log\gamma\in   (L^\infty\cap\mathcal{E}')(\R^n)$, and notice that $A\in H^{1/2}(\R^n)$. To see the latter property, we write $A=\gamma^{-1}\nabla(\gamma-1)$. Here $\nabla(\gamma-1)\in H^{1/2}(\R^n)$ and this space is stable under multiplication by bounded Lipschitz continuous functions on $\R^n$.  

Under our improved regularity assumptions, we shall now get sharpened estimates for the remainder in \eqref{eq_prop_CGO}. To this end, we only need to re-examine the estimate for $hI$, where $I$ is given in \eqref{eq_2_10_5_1}.  Now using  \eqref{eq_2_10_7} and  \eqref{eq_2_10_8}, we have
\begin{equation}
\label{eq_2_10_11}
\begin{aligned}
hI=\bigg|h\int_{\Omega} a A\cdot \nabla v dx\bigg|&\le \mathcal{O}(h)\|A-A_\tau\|_{L^2(\Omega)}\|\nabla v\|_{L^2(\Omega)}\\
&+h\|a \nabla A_\tau+ A_\tau\cdot \nabla a\|_{L^2(\Omega)}\|v\|_{L^2(\Omega)}\\
&\le (o(\tau^{1/2})+ho(\tau^{-1/2})+ \mathcal{O}(h))\|v\|_{H^1_{\text{scl}}(\Omega)}.
\end{aligned}
\end{equation}
Choosing $\tau=h$, we obtain from \eqref{eq_2_10_4}, \eqref{eq_2_10_5} and \eqref{eq_2_10_11} that $\|r\|_{H^1_{\text{scl}}(\Omega)}=o(h^{1/2})$ as $h\to 0$. 

Thus,  we have the following result. 
\begin{prop}
\label{prop_CGO_solutions}
Let $\Omega \subset\R^n$, $n\ge 3$, be  a bounded open set with $C^2$ boundary and let $\tilde \Omega\subset \R^n$ be an open set such that  $\Omega\subset\subset \tilde \Omega$.  Let $\gamma\in W^{1,\infty}(\Omega)\cap H^{\frac{3}{2}}(\Omega)$ and $\gamma>0$ on $\overline{\Omega}$.  Then for all $h>0$ small enough, there exists a solution $u(x;h)\in H^1(\Omega)$ to the conductivity equation $L_\gamma u=0$ in $\Omega$, of the form
\[
u(x;h)=\gamma^{-1/2}e^{\frac{\varphi+i\psi}{h}}(a(x)+r(x;h)),
\]
where $\varphi\in C^\infty(\tilde \Omega,\R)$ is a limiting Carleman weight for the semiclassical Laplacian on $\tilde \Omega$, $\psi\in C^\infty(\tilde \Omega,\R)$ is a solution to the eikonal equation \eqref{eq_2_eikonal}, $a\in C^\infty(\overline{\Omega})$ is a solution of the first transport equation \eqref{eq_2_10_3}, and the remainder term $r$ is such that $\|r\|_{H^1_{\emph{\text{scl}}}(\Omega)}=o(h^{1/2})$ as $h\to 0$. 
\end{prop}

\section{Boundary Carleman estimates}
\label{sec_BCE}

\subsection{Boundary Carleman estimates for the Laplacian}
\label{sec_BCE_Laplace}

The following result is an extension of \cite[Proposition 2.3]{DKSU_2007} valid for functions which need not vanish along the boundary of $\Omega$.  
\begin{prop}
Let $\Omega\subset\R^n$, $n\ge 3$, be a bounded open set with  $C^2$ boundary. Let $\varphi\in C^\infty(\tilde \Omega)$, $\Omega\subset\subset\tilde \Omega\subset\R^n$, be a limiting Carleman weight for the semiclassical Laplacian, and set $\tilde \varphi=\varphi+\frac{h}{2\varepsilon}\varphi^2$, $\varepsilon>0$. Then for all $0<h\ll \varepsilon\ll 1$, we have 
\begin{equation}
\label{bc_-1}
\begin{aligned}
&\mathcal{O}(h)\|v\|^2_{L^2(\p \Omega)}+\mathcal{O}(h^3)\int_{\p\Omega}|\p_\nu v||v|dS
+\mathcal{O}(h^3)\int_{\p\Omega_+} (\p_\nu\tilde \varphi)|\p_\nu v|^2dS \\
&+\mathcal{O}(h^3)\|\nabla_t v\|^2_{L^2(\p \Omega)}
+\mathcal{O}(h^3)\int_{\p \Omega} |\nabla_t v||\p_\nu v|dS
\\
&+\mathcal{O}(1)\| e^{\frac{\tilde \varphi}{h}} (-h^2\Delta)(e^{-\frac{\tilde\varphi}{h}} v) \|^2_{L^2(\Omega)}
  \ge \frac{h^2}{\varepsilon}\|v\|^2_{H^1_{\emph{\text{scl}}}(\Omega)} -h^3\int_{\p\Omega_-} (\p_\nu\tilde \varphi)|\p_\nu v|^2dS,
\end{aligned}
\end{equation}
for all $v\in H^2(\Omega)$. Here $\nu$ is the unit outer normal to $\p\Omega$, $\nabla_t$ is the tangential component of the gradient, and 
\[
\p\Omega_\pm=\{x\in \p\Omega: \pm \p_\nu\varphi(x)  \ge 0\}.
\]

\end{prop}

\begin{proof}
By  density, it suffices to establish \eqref{bc_-1} for $v\in C^\infty (\overline{\Omega})$.

We have 
\[
e^{\frac{\tilde \varphi}{h}}\circ hD\circ e^{-\frac{\tilde\varphi}{h}}=hD+i\nabla \tilde \varphi,
\]
and therefore, 
\[
e^{\frac{\tilde \varphi}{h}}\circ (-h^2\Delta)\circ e^{-\frac{\tilde\varphi}{h}}=\tilde P+i\tilde Q,
\]
where $\tilde P$ and $\tilde Q$ are the formally  self-adjoint operators given by
\begin{equation}
\label{bc_0}
\begin{aligned}
\tilde P&=-h^2\Delta -(\nabla \tilde \varphi)^2,\\
\tilde Q&=\frac{2h}{i}\nabla \tilde \varphi\cdot \nabla +\frac{h}{i}\Delta\tilde \varphi.
\end{aligned}
\end{equation} 

We write
\begin{equation}
\label{bc_1}
\|(\tilde P+i\tilde Q)v\|_{L^2(\Omega)}^2=\|\tilde P v\|^2_{L^2(\Omega)}+\|\tilde Q u\|^2_{L^2(\Omega)}+i (\tilde Qv,\tilde P v)_{L^2(\Omega)}-i(\tilde P v,\tilde Qv)_{L^2(\Omega)}.
\end{equation}

Integrating by parts, we get
\begin{equation}
\label{bc_2}
\begin{aligned}
(\tilde P v,\tilde Q v)_{L^2(\Omega)}&=\int_{\Omega} ih \tilde P v (2 \nabla \tilde \varphi \cdot \nabla \overline{v}+(\Delta \tilde \varphi) \overline{v} )dx\\
&=(\tilde Q\tilde P v,v)_{L^2(\Omega)}+2ih(\p_\nu \tilde \varphi \tilde Pv,v)_{L^2(\p \Omega)},
\end{aligned}
\end{equation}
and
\begin{equation}
\label{bc_3}
\begin{aligned}
(\tilde Q v,\tilde P v)_{L^2(\Omega)}&=\int_{\Omega} \tilde Q v(-h^2\Delta \overline{v}-(\nabla \tilde \varphi)^2\overline{v})dx\\
&=(\tilde P \tilde Q v,v )_{L^2(\Omega)}-h^2(\tilde Q v, \p_\nu v)_{L^2(\p \Omega)}+h^2(\p_\nu(\tilde Qv), v)_{L^2(\p \Omega)}.
\end{aligned}
\end{equation}

Substituting \eqref{bc_2} and \eqref{bc_3} into \eqref{bc_1},
 we obtain that 
 \begin{equation}
\label{bc_4}
 \|(\tilde P+i\tilde Q)v\|^2_{L^2(\Omega)}=\|\tilde Pv\|^2_{L^2(\Omega)}+\|\tilde Q u\|^2_{L^2(\Omega)}+i([\tilde P,\tilde Q]v,v)_{L^2(\Omega)}+BT_1,
 \end{equation} 
 where
  \begin{equation}
\label{bc_4_1}
 BT_1=ih^2(\p_\nu(\tilde Qv), v)_{L^2(\p \Omega)}-i h^2(\tilde Q v, \p_\nu v)_{L^2(\p \Omega)}+2h(\p_\nu \tilde \varphi \tilde Pv,v)_{L^2(\p \Omega)}.
\end{equation}

Let us now understand the boundary terms $BT_1$ in \eqref{bc_4_1}. In doing so, we shall use the following expression for the Laplacian and the gradient on the boundary, see \cite[p. 16]{Choulli_book}, \cite{Lee_Uhlmann}, 
\begin{equation}
\label{bc_5}
\begin{aligned}
\Delta v&=\Delta_t v+H \p_\nu v+\p_\nu^2 v \quad \text{on}\quad \p \Omega,\\
\nabla v&=(\p_\nu v)\nu +\nabla_t v\quad \text{on}\quad \p \Omega,
\end{aligned}
\end{equation}
where $\Delta_t$ and $\nabla_t$ are the tangential Laplacian and gradient on $\p \Omega$, and $H\in C(\p \Omega)$.  First using \eqref{bc_0} and \eqref{bc_5}, we write
\begin{equation}
\label{bc_6}
\begin{aligned}
\tilde P&=-h^2\Delta_t-h^2H\p_\nu-h^2\p_\nu^2-(\nabla \tilde \varphi)^2 \quad \text{on}\quad \p \Omega,\\
\tilde Q&=\frac{2h}{i}(\p_\nu \tilde \varphi)\p_\nu+\frac{2h}{i}A_t+\frac{h}{i}\Delta\tilde \varphi \quad \text{on}\quad \p \Omega,
\end{aligned}
\end{equation}
where the vector field 
\begin{equation}
\label{bc_6_1}
A_t=\nabla_t\tilde \varphi \cdot \nabla_t
\end{equation}
 is  real tangential. Thus, 
\begin{equation}
\label{bc_7}
\p_\nu(\tilde Qv)=\frac{2h}{i}(\p^2_\nu\tilde \varphi)\p_\nu v+\frac{2h}{i}(\p_\nu\tilde \varphi)\p^2_\nu v +\frac{2h}{i}\p_\nu(A_t v)+\frac{h}{i}\p_\nu( v \Delta\tilde \varphi).
\end{equation}

Substituting \eqref{bc_6} and \eqref{bc_7} into \eqref{bc_4_1}, and using the fact that the terms containing  $\p^2_\nu v$ cancel out, we get 
\begin{equation}
\label{bc_7_0_1}
\begin{aligned}
BT_1=&\big( [h^3 \p_\nu(\Delta \tilde\varphi) -2h(\p_\nu \tilde \varphi) (\nabla\tilde \varphi)^2] v ,v\big)_{L^2(\p \Omega)}\\
&+h^3\big( [2\p_\nu ^2 \tilde \varphi +\Delta\tilde \varphi-2H \p_\nu \tilde \varphi]\p_\nu v,v\big)_{L^2(\p \Omega)}- h^3\big( (\Delta\tilde \varphi) v,\p_\nu v\big)_{L^2(\p \Omega)}\\
&-2h^3\big( (\p_\nu \tilde \varphi)\p_\nu v,\p_\nu v \big)_{L^2(\p \Omega)}\\
& +h^3 \big[ 2\big(\p_\nu(A_t v), v \big)_{L^2(\p \Omega)} -2\big(A_t v, \p_\nu v \big)_{L^2(\p \Omega)} +2\big(\nabla_t v, \nabla_t ((\p_\nu \tilde \varphi) v) \big)_{L^2(\p \Omega)}   \big].
\end{aligned}
\end{equation}
Here in the last term we have performed an integration by parts using the tangential  Laplacian.

In order not to have second derivatives of $v$ on the boundary $\p \Omega$, integrating by parts in the fifth term in \eqref{bc_7_0_1} and using \eqref{bc_6_1}, we obtain that 
\begin{equation}
\label{bc_7_0_2}
\begin{aligned}
 \big(\p_\nu(A_t v), v \big)_{L^2(\p \Omega)}&=\big(A_t \p_\nu v, v \big)_{L^2(\p \Omega)}+\big(X v, v \big)_{L^2(\p \Omega)}\\
 &=-\big( \p_\nu v, A_t v \big)_{L^2(\p \Omega)}-\big(\p_\nu v, (\div A_t) v \big)_{L^2(\p \Omega)}+\big(X v, v \big)_{L^2(\p \Omega)}\\
 &=-\big( \p_\nu v, A_t v \big)_{L^2(\p \Omega)}-\big(\p_\nu v, (\Delta_t\tilde \varphi) v \big)_{L^2(\p \Omega)}+\big(X v, v \big)_{L^2(\p \Omega)}.
\end{aligned}
\end{equation}
Here the vector field $X$ is given by 
\[
X=[\p_\nu,A_t].
\]

Let us now consider the non-boundary terms in \eqref{bc_4}. To understand the term $i([\tilde P, \tilde Q]v,v)_{L^2(\Omega)}$, we recall  from \cite[p. 473]{DKSU_2007} that
\begin{equation}
\label{bc_7_1}
\begin{aligned}
i[\tilde P, \tilde Q]=\frac{4h^2}{\varepsilon}\bigg(1+\frac{h}{\varepsilon}\varphi \bigg)^2(\nabla\varphi)^4+\frac{h}{2}(a\tilde P+\tilde P a)+\frac{h}{2}(b^w\tilde Q+\tilde Q b^w)+h^3c(x),
\end{aligned}
\end{equation}
where 
\begin{equation}
\label{bc_7_2}
a(x)=\frac{4h}{\varepsilon}(\nabla \tilde \varphi)^2-4\frac{\tilde \varphi''\nabla \tilde \varphi \cdot \nabla \tilde \varphi}{(\nabla\tilde \varphi)^2},
\end{equation}
\[
b(x,\xi)=\lambda(x)\cdot\xi,
\]
with $\lambda$ being a real $C^\infty$ vector field, and $c\in C^\infty(\overline{\Omega})$. 
Here also 
\begin{equation}
\label{bc_8}
b^w=\frac{1}{2}(\lambda(x)\circ hD_x+hD_x\circ\lambda(x))=\lambda(x)\cdot hD_x+\frac{h}{2i}\div\lambda,
\end{equation}
is the semiclassical Weyl quantization of $b$.

Using the fact that  $4 (1+\frac{h}{\varepsilon}\varphi )^2(\nabla\varphi)^4\ge 1/C$ on $\overline{\Omega}$, we obtain from \eqref{bc_7_1} that for $0<h\ll \varepsilon$ small enough,  
\begin{equation}
\label{bc_8_1}
\begin{aligned}
i([\tilde P,& \tilde Q]v,v)_{L^2(\Omega)} \ge \frac{h^2}{C\varepsilon}\|v\|^2_{L^2(\Omega)}+ \frac{h}{2}\big( (a\tilde P+\tilde P a)v, v\big)_{L^2(\Omega)}\\
&+ \frac{h}{2}\big( (b^w\tilde Q+\tilde Q b^w) v, v\big)_{L^2(\Omega)} +h^3(cv,v)_{L^2(\Omega)}\\
&\ge \frac{h^2}{C\varepsilon}\|v\|^2_{L^2(\Omega)}+ \frac{h}{2}\big( (a\tilde P+\tilde P a)v, v\big)_{L^2(\Omega)}+ \frac{h}{2}\big( (b^w\tilde Q+\tilde Q b^w) v, v\big)_{L^2(\Omega)}.
\end{aligned}
\end{equation}

Integrating by parts, we get 
\[
(\tilde P(av),v)_{L^2(\Omega)}=(v, a\tilde Pv)_{L^2(\Omega)}-h^2(\p_\nu (av),v)_{L^2(\p \Omega)}
+h^2(av,\p_\nu v)_{L^2(\p \Omega)},
\] 
and therefore,
\begin{equation}
\label{bc_9}
\frac{h}{2}\big( (a\tilde P+\tilde P a)v, v\big)_{L^2(\Omega)}=h\Re (a\tilde Pv, v)_{L^2(\Omega)}+ BT_2,
\end{equation}
where 
\begin{equation}
\label{bc_9_1}
BT_2=-\frac{h^3}{2}(\p_\nu (av),v)_{L^2(\p \Omega)}
+\frac{h^3}{2}(av,\p_\nu v)_{L^2(\p \Omega)}.
\end{equation}
Using \eqref{bc_8}, and integrating by parts, we obtain that 
\begin{equation}
\label{bc_10}
\begin{aligned}
(&b^w \tilde Qv, v)_{L^2(\Omega)}=(\tilde Q v, b^w v)_{L^2(\Omega)}+\frac{h}{i}\int_{\p\Omega}(\tilde Q v)\lambda\cdot \nu \overline{v}dS=(\tilde Q v, b^w v)_{L^2(\Omega)}\\
&-2h^2\big( \p_\nu \tilde \varphi\p_\nu v, \lambda\cdot \nu v\big)_{L^2(\p \Omega)}- 2h^2\big(A_t v, \lambda\cdot \nu v\big)_{L^2(\p \Omega)}
-h^2\big( (\Delta \tilde \varphi) v, \lambda\cdot \nu v\big)_{L^2(\p \Omega)}.
\end{aligned}
\end{equation}
Here we have also used \eqref{bc_6}. 

In view of \eqref{bc_0}, and another integration by parts, we get
\begin{equation}
\label{bc_11}
\begin{aligned}
\big( \tilde Q(b^w v), &v\big)_{L^2(\Omega)}= \big(b^w v,  \tilde Q v\big)_{L^2(\Omega)}
-2ih \int_{\p \Omega}(\p_\nu \tilde \varphi)(b^w v)\overline{v}dS\\
&=\big(b^w v,  \tilde Q v\big)_{L^2(\Omega)}
-h^2\big( (\p_\nu\tilde \varphi \div \lambda)v, v\big)_{L^2(\p\Omega)}
-2h^2\big( \p_\nu\tilde \varphi \lambda\cdot \nabla v, v\big)_{L^2(\p\Omega)}
\end{aligned}
\end{equation}
Here we have also used \eqref{bc_8}. 

It follows from \eqref{bc_10} and \eqref{bc_11} that 
\begin{equation}
\label{bc_12}
\frac{h}{2}\big( (b^w\tilde Q+\tilde Q b^w) v, v\big)_{L^2(\Omega)}=h\Re (\tilde Q v, b^w v)_{L^2(\Omega)}+BT_3,
\end{equation}
where 
\begin{equation}
\label{bc_12_1}
\begin{aligned}
&BT_3=-h^3\big( \p_\nu \tilde \varphi\p_\nu v, \lambda\cdot \nu v\big)_{L^2(\p \Omega)}- h^3\big(A_t v, \lambda\cdot \nu v\big)_{L^2(\p \Omega)}\\
&-\frac{h^3}{2}\big( (\Delta \tilde \varphi) v, \lambda\cdot \nu v\big)_{L^2(\p \Omega)}
-\frac{h^3}{2}\big( (\p_\nu\tilde \varphi \div \lambda)v, v\big)_{L^2(\p\Omega)}
-h^3\big( \p_\nu\tilde \varphi \lambda\cdot \nabla v, v\big)_{L^2(\p\Omega)}.
\end{aligned}
\end{equation}

In view of \eqref{bc_4}, \eqref{bc_8_1}, \eqref{bc_9} and \eqref{bc_12}, we get
\begin{equation}
\label{bc_13}
\begin{aligned}
 \|(\tilde P+i\tilde Q)v\|^2_{L^2(\Omega)}\ge \|\tilde Pv\|^2_{L^2(\Omega)}+\|\tilde Q v\|^2_{L^2(\Omega)}+ \frac{h^2}{C\varepsilon}\|v\|^2_{L^2(\Omega)}\\
 +h\Re (a\tilde Pv, v)_{L^2(\Omega)} +h\Re (\tilde Q v, b^w v)_{L^2(\Omega)}
  +BT_4,
\end{aligned}
\end{equation}
where 
\begin{equation}
\label{bc_13_1}
BT_4=BT_1+BT_2+BT_3.
\end{equation}

In view of \eqref{bc_7_2}, we see that $|a|=\mathcal{O}(1)$ uniformly in $0<h\ll\varepsilon\ll 1$. Using this and Peter--Paul's inequality, we have
\begin{equation}
\label{bc_14}
h| (a\tilde Pv, v)_{L^2(\Omega)}|\le \mathcal{O}(h)\|\tilde P v\|_{L^2(\Omega)}\|v\|_{L^2(\Omega)} \le \frac{1}{2}\|\tilde P v\|_{L^2(\Omega)}^2+\mathcal{O}(h^2)\|v\|^2_{L^2(\Omega)}.
\end{equation}

Using the fact that $b^w: H^1_{\text{scl}}(\Omega)\to L^2(\Omega)$ is bounded uniformly in $h$, cf. \eqref{bc_8}, we get for all $0<h\ll\varepsilon$ small enough, 
\begin{equation}
\label{bc_15}
\begin{aligned}
h| (\tilde Q v, b^w v)_{L^2(\Omega)}|\le \mathcal{O}(h)\|\tilde Q v\|_{L^2(\Omega)}\big(\|v\|_{L^2(\Omega)}+\|h\nabla v\|_{L^2(\Omega)}\big)\\
\le  \frac{1}{2}\|\tilde Q v\|_{L^2(\Omega)} + \mathcal{O}(h^2) \|h\nabla v\|^2_{L^2(\Omega)}+\mathcal{O}(h^2)\|v\|_{L^2}^2.
\end{aligned}
\end{equation}
In the last estimate we have used Peter--Paul's inequality.

It follows from \eqref{bc_13}, \eqref{bc_14} and \eqref{bc_15} that for all $0<h\ll\varepsilon$ small enough, 
\begin{equation}
\label{bc_16}
\begin{aligned}
 \|(\tilde P+i\tilde Q)v\|^2_{L^2(\Omega)}\ge& \frac{1}{2}\|\tilde Pv\|^2_{L^2(\Omega)}+\frac{1}{2}\|\tilde Q v\|^2_{L^2(\Omega)}\\
 &+ \frac{h^2}{C\varepsilon}\|v\|^2_{L^2(\Omega)}-\mathcal{O}(h^2) \|h\nabla v\|^2_{L^2(\Omega)}+ BT_4.
\end{aligned}
\end{equation}

Using \eqref{bc_0} and integrating by parts, we have
\[
(\tilde P v, v)_{L^2(\Omega)}=\|h\nabla v\|^2_{L^2(\Omega)}-h^2\int_{\p \Omega}\p_\nu v\overline{v}dS- \|(\nabla \tilde \varphi)v\|^2_{L^2(\Omega)},
\]
and hence, 
\begin{equation}
\label{bc_17}
\|h\nabla v\|^2_{L^2(\Omega)}\le C\big( \|\tilde Pv\|^2_{L^2(\Omega)}+\|v\|^2_{L^2(\Omega)}+h^2|(\p_\nu v,v)_{L^2(\p \Omega)}|\big).
\end{equation}

It follows from \eqref{bc_16} and \eqref{bc_17} that for all $0<h\ll\varepsilon$ small enough, 
\begin{equation}
\label{bc_18}
\begin{aligned}
 \|(\tilde P+i\tilde Q)v\|^2_{L^2(\Omega)}&\ge  \bigg(\frac{1}{2}-\frac{h^2}{2C\varepsilon}\bigg)\|\tilde Pv\|^2_{L^2(\Omega)}\\
 &+ 
 \frac{h^2}{2C\varepsilon}\bigg(\frac{1}{C}\|h\nabla v\|^2_{L^2(\Omega)} - \|v\|^2_{L^2(\Omega)}-h^2|(\p_\nu v,v)_{L^2(\p \Omega)}|\bigg) \\
 &+ \frac{h^2}{C\varepsilon}\|v\|^2_{L^2(\Omega)}-\mathcal{O}(h^2) \|h\nabla v\|^2_{L^2(\Omega)}+ BT_4\\
& \ge \frac{h^2}{C\varepsilon}\|v\|^2_{H^1_{\text{scl}}(\Omega)} +BT_5,
\end{aligned}
\end{equation}
where 
\begin{equation}
\label{bc_19}
BT_5=BT_4-\mathcal{O}(h^3)|(\p_\nu v,v)_{L^2(\p \Omega)}|.
\end{equation}

Let us now understand the boundary terms $BT_4$. Using \eqref{bc_7_0_1}, \eqref{bc_7_0_2}, \eqref{bc_9_1}, \eqref{bc_12_1},  and \eqref{bc_13_1}, we get
\begin{equation}
\label{bc_20}
\begin{aligned}
&BT_4=\bigg( \bigg[-2h(\p_\nu \tilde \varphi) (\nabla\tilde \varphi)^2 +h^3 \p_\nu(\Delta \tilde\varphi)\\
&-\frac{h^3}{2}\p_\nu a -\frac{h^3}{2}(\lambda\cdot\nu)(\Delta\tilde \varphi)-\frac{h^3}{2}(\p_\nu \tilde \varphi)\div\lambda \bigg] v ,v\bigg)_{L^2(\p \Omega)}\\
&+h^3\big( [2\p_\nu ^2 \tilde \varphi +\Delta\tilde \varphi-2H \p_\nu \tilde \varphi-2(\Delta_t\tilde \varphi) +2X\cdot\nu -\frac{a}{2}-2(\p_\nu \tilde \varphi)(\lambda\cdot\nu) ]\p_\nu v,v\big)_{L^2(\p \Omega)}\\
&+h^3\big( [-\Delta\tilde \varphi+\frac{a}{2}] v,\p_\nu v\big)_{L^2(\p \Omega)}\\
&-2h^3\big( (\p_\nu \tilde \varphi)\p_\nu v,\p_\nu v \big)_{L^2(\p \Omega)}\\
& +h^3 \bigg[ -2\big( \p_\nu v, A_t v \big)_{L^2(\p \Omega)}+ 2 \big(X_t v,v\big)_{L^2(\p \Omega)} -2\big(A_t v, \p_\nu v \big)_{L^2(\p \Omega)} \\
&+2\big(\nabla_t v, (\p_\nu \tilde \varphi) \nabla_t v \big)_{L^2(\p \Omega)}  +2\big(\nabla_t v, (\nabla_t \p_\nu \tilde \varphi) v \big)_{L^2(\p \Omega)} - \big((\lambda\cdot \nu)A_t v,  v\big)_{L^2(\p \Omega)}\\
&-((\p_\nu\tilde \varphi)\lambda_t v,v)_{L^2(\p\Omega)} \bigg].
\end{aligned}
\end{equation}
Here we have used that $X=X_t+X\cdot \nu\p_\nu$ and $\lambda\cdot\nabla=\lambda_t+\lambda\cdot\nu\p_\nu$, where  $X_t$ and $\lambda_t$ are tangential vector fields.

Putting the boundary terms $BT_5$ to the left hand side of \eqref{bc_18}, and using  that 
\[
|(Y_t v,v)_{L^2(\p \Omega)}|\le C\|\nabla_t v\|_{L^2(\p \Omega)}\|v\|_{L^2(\p \Omega)}\le \frac{C}{2}(\|\nabla_t v\|^2_{L^2(\p \Omega)}+\|v\|^2_{L^2(\p \Omega)}),
\]
where $Y_t$ is a tangential vector field, 
together with  \eqref{bc_18},
\eqref{bc_19} and \eqref{bc_20}, we obtain \eqref{bc_-1}. The proof is complete. 
\end{proof}

\subsection{Boundary Carleman estimates for a first order perturbation of the Laplacian}
\label{subsection_BCE}

In this subsection we shall establish a boundary Carleman estimate for the operator 
\[
-\Delta +A\cdot \nabla+V 
\]
where $A\in L^\infty(\Omega, \C^n)$,  $V\in L^\infty(\Omega, \C)$ are possibly $h$--dependent with 
\begin{equation}
\label{eq_estimates_A_V}
\|A\|_{L^\infty(\Omega)}=\mathcal{O}(1), \quad \|V\|_{L^\infty(\Omega)}=\mathcal{O}\bigg(\frac{1}{h}\bigg),
\end{equation}
as  $h\to 0$. 
This estimate will play a crucial role in getting rid of the boundary terms over the unaccessible part of the boundary in the next section.

\begin{prop}
\label{prop_boundary_Carleman} The following Carleman estimate 
\begin{equation}
\label{eq_2_11_7}
\begin{aligned}
\mathcal{O}(h)&\|e^{-\frac{\varphi}{h}}u\|^2_{L^2(\p \Omega)}+\mathcal{O}(h^2)\int_{\p \Omega} e^{-\frac{2\varphi}{h}}|\p_\nu u||u|dS \\
&+\mathcal{O}(h^3)\int_{\p \Omega_{-}}(-\p_\nu \varphi)e^{-\frac{2\varphi}{h}}|\p_\nu u|^2dS\\
&+\mathcal{O}(h^3)\|e^{-\frac{\varphi}{h}}\nabla_t u\|^2_{L^2(\p \Omega)} 
+\mathcal{O}(h^3)\int_{\p \Omega} e^{-\frac{2\varphi}{h}}|\nabla_t u||\p_\nu u|dS
\\
&+\mathcal{O}(1)\| e^{-\varphi/h} (-h^2\Delta + hA\cdot h\nabla +h^2V) u\|^2_{L^2(\Omega)}\\
&\ge h^2( \|e^{-\frac{\varphi}{h}}u\|^2_{L^2(\Omega)}+ \|e^{-\frac{\varphi}{h}}h\nabla u\|^2_{L^2(\Omega)}) + h^3\int_{\p \Omega_{+}}(\p_\nu \varphi)e^{-\frac{2\varphi}{h}}|\p_\nu u|^2dS 
\end{aligned}
\end{equation}
holds for all $u\in H^2(\Omega)$ and all $h>0$ small enough. 
\end{prop}

\begin{proof}
We have 
\begin{equation}
\label{eq_2_11_9}
\begin{aligned}
e^{\frac{\tilde \varphi}{h}}\circ &(-h^2\Delta +hA\cdot h\nabla +h^2V)\circ e^{-\frac{ \tilde \varphi }{h}}\\
&=e^{\frac{\tilde \varphi}{h}}\circ(-h^2\Delta)\circ e^{-\frac{ \tilde \varphi}{h}}+
hA\cdot h\nabla -hA\cdot \nabla \tilde \varphi+h^2V.
\end{aligned}
\end{equation}
Let $v\in H^2(\Omega)$.  Using that  $\|\nabla \tilde \varphi\|_{L^\infty(\Omega)}\le \mathcal{O}(1)$,  \eqref{eq_estimates_A_V}, we get  
\begin{equation}
\label{bcf_0}
\|hA\cdot h\nabla v -hA\cdot (\nabla \tilde \varphi) v+h^2Vv \|_{L^2(\Omega)}\le \mathcal{O}(h)\|v\|_{H^1_{\text{scl}}(\Omega)}. 
\end{equation}
Using that $a^2\le 2((a+b)^2+b^2)$, $a,b>0$, and \eqref{bcf_0}, we obtain from \eqref{bc_-1} that for all $0<h\ll \varepsilon\ll 1$, 
\begin{equation}
\label{bcf_1}
\begin{aligned}
\mathcal{O}(h)&\|v\|^2_{L^2(\p \Omega)}+\mathcal{O}(h^3)\int_{\p\Omega}|\p_\nu v||v|dS\\
&+\mathcal{O}(h^3)\int_{\p\Omega_+} (\p_\nu\tilde \varphi)|\p_\nu v|^2dS+ 
\mathcal{O}(h^3)\|\nabla_t v\|^2_{L^2(\p \Omega)}
+\mathcal{O}(h^3)\int_{\p \Omega} |\nabla_t v||\p_\nu v|dS
\\
&+\mathcal{O}(1)\| e^{\frac{\tilde \varphi}{h}}(-h^2\Delta +hA\cdot h\nabla +h^2V)(e^{-\frac{ \tilde \varphi }{h}}v) \|^2_{L^2(\Omega)}\\
 & \ge \frac{h^2}{\varepsilon}\|v\|^2_{H^1_{\text{scl}}(\Omega)}-h^3\int_{\p\Omega_-} (\p_\nu\tilde \varphi)|\p_\nu v|^2dS. 
\end{aligned}
\end{equation}

Let us now take $\varepsilon>0$ to be small but fixed and let $v=e^{\frac{\varphi^2}{2\varepsilon}}e^{\frac{\varphi}{h}}u$.  Using that 
\[
1\le e^{\frac{\varphi^2}{2\varepsilon}}\le C\quad \text{on}\quad \overline{\Omega},
\]
and for all $h$ small enough, 
\[
\frac{1}{2}\le \frac{\p_\nu\tilde\varphi}{\p_\nu \varphi}\le \frac{3}{2} \quad \text{on}\quad \overline{\Omega}
\]
and
\[
|\nabla(e^{\frac{\varphi^2}{2\varepsilon}}e^{\frac{\varphi}{h}})|
=e^{\frac{\varphi^2}{2\varepsilon}}e^{\frac{\varphi}{h}}\frac{1}{h}\bigg|\bigg(1+\frac{h}{\varepsilon}\varphi\bigg)\nabla\varphi\bigg|\le \mathcal{O}\bigg(\frac{1}{h}\bigg)e^{\frac{\varphi}{h}} \quad \text{on}\quad \overline{\Omega},
\]
we obtain that 
\begin{equation}
\label{bcf_3}
\begin{aligned}
\mathcal{O}(h)&\|e^{\frac{\varphi}{h}}u\|^2_{L^2(\p \Omega)}+\mathcal{O}(h^2)\int_{\p \Omega} e^{\frac{2\varphi}{h}}|\p_\nu u||u|dS +\mathcal{O}(h^3)\int_{\p \Omega_{+}}(\p_\nu \varphi)e^{\frac{2\varphi}{h}}|\p_\nu u|^2dS\\
&+\mathcal{O}(h^3)\|e^{\frac{\varphi}{h}}\nabla_t u\|^2_{L^2(\p \Omega)} 
+\mathcal{O}(h^3)\int_{\p \Omega} e^{\frac{2\varphi}{h}}|\nabla_t u||\p_\nu u|dS\\
&+\mathcal{O}(1)\| e^{\varphi/h} (-h^2\Delta + hA\cdot h\nabla +h^2V) u\|^2_{L^2(\Omega)}\\
&\ge h^2( \|e^{\frac{\varphi}{h}}u\|^2_{L^2(\Omega)}+ \|e^{\frac{\varphi}{h}}h\nabla u\|^2_{L^2(\Omega)}) + h^3\int_{\p \Omega_{-}}(-\p_\nu \varphi)e^{\frac{2\varphi}{h}}|\p_\nu u|^2dS. 
\end{aligned}
\end{equation}
Here we have also used that for $\tau>1$ sufficiently large but fixed, we get
\begin{align*}
\|e^{\frac{\varphi^2}{2\varepsilon}}e^{\frac{\varphi}{h}}u\|^2_{H^1_{\text{scl}}(\Omega)}&\ge \|e^{\frac{\varphi}{h}}u\|^2_{L^2(\Omega)}+\frac{1}{\tau} \|h\nabla(e^{\frac{\varphi^2}{2\varepsilon}}e^{\frac{\varphi}{h}})u+ e^{\frac{\varphi^2}{2\varepsilon}}e^{\frac{\varphi}{h}}h\nabla u \|^2_{L^2(\Omega)}\\
&\ge \|e^{\frac{\varphi}{h}}u\|^2_{L^2(\Omega)}+ \frac{1}{2\tau}\|e^{\frac{\varphi}{h}}h\nabla u\|^2_{L^2(\Omega)}-\mathcal{O}\bigg(\frac{1}{\tau}\bigg) \|e^{\frac{\varphi}{h}}u\|^2_{L^2(\Omega)}\\
&\ge \frac{1}{C}(\|e^{\frac{\varphi}{h}}u\|^2_{L^2(\Omega)}+ \|e^{\frac{\varphi}{h}}h\nabla u\|^2_{L^2(\Omega)}).
\end{align*}

Replacing $\varphi$ by $-\varphi$ in \eqref{bcf_3}, we get  \eqref{eq_2_11_7}. The proof is complete. 
\end{proof}

\section{Integral identity}
\label{sec_integral identity}
The following result is due to \cite[Theorem 7]{Brown_1996}. Since we need the integral identity with boundary terms, we present the proof for convenience of the reader. 
\begin{lem}
\label{lem_integral_identity}
Let $\Omega\subset \R^n$, $n\ge 3$, be a bounded open set with Lipschitz boundary, and let $\gamma_1,\gamma_2\in W^{1,\infty}(\Omega)$ be such that $\gamma_1,\gamma_2>0$ on $\overline{\Omega}$, and 
$\gamma_1=\gamma_2$ on $\p \Omega$. Let $u_j\in H^1(\Omega)$ satisfy $L_{\gamma_j}u_j=0$ in $\Omega$, $j=1,2$, and let $\tilde u_1\in H^1(\Omega)$ satisfy $L_{\gamma_1}\tilde u_1=0$ in $\Omega$ with $\tilde u_1=u_2$ on $\p \Omega$. Then 
\begin{equation}
\label{eq_integral_identity}
\int_{\Omega} \big[ -\nabla \gamma_1^{1/2} \cdot \nabla (\gamma_2^{1/2}u_1 u_2)+\nabla \gamma_2^{1/2}\cdot \nabla(\gamma_1^{1/2}u_1u_2) \big]dx=\int_{\p \Omega} (\Lambda_{\gamma_1}\tilde u_1 -\Lambda_{\gamma_2}u_2)u_1dS,
\end{equation}
where the integral over $\p\Omega$ is understood in the sense of the dual pairing between $H^{-1/2}(\p \Omega)$ and $H^{1/2}(\p \Omega)$. 
\end{lem}

\begin{proof}

Set $v_j=\gamma_j^{1/2}u_j$, $j=1,2$. We know that if $u_1\in H^1(\Omega)$ satisfies $L_{\gamma_1}u_1=0$ in $\Omega$, then the trace $\gamma_1\p_\nu u_1\in H^{-1/2}(\p \Omega)$, and we have the following integration by parts formula, 
\begin{align*}
0=\int_{\Omega} (\nabla \cdot (\gamma_1\nabla u_1))(\gamma_1^{-1/2}v_2)dx=&-\int_{\Omega} \gamma_1\nabla u_1\cdot \nabla (\gamma_1^{-1/2}v_2)dx\\
&+\int_{\p \Omega} (\Lambda_{\gamma_1}u_1)u_2dS.
\end{align*}
Here we have used the fact that $\gamma_1=\gamma_2$ on $\p \Omega$. Thus, using that $\nabla \gamma_1^{-1/2}=-\gamma_1^{-1}\nabla \gamma_1^{1/2}$, we get
\begin{equation}
\label{eq_2_10_12}
\begin{aligned}
\int_{\p \Omega} (\Lambda_{\gamma_1}u_1)u_2dS&=\int_{\Omega} \gamma_1\nabla (\gamma_1^{-1/2} v_1)\cdot \nabla (\gamma_1^{-1/2}v_2)dx\\
&=\int_{\Omega} [-\nabla \gamma_1^{1/2}\cdot \nabla (\gamma_1^{-1/2}v_1v_2)+\nabla v_1\cdot\nabla v_2]dx.
\end{aligned}
\end{equation}
Similarly, we have 
\begin{equation}
\label{eq_2_10_13}
\int_{\p \Omega} (\Lambda_{\gamma_2}u_2)u_1dS=\int_{\Omega} [-\nabla \gamma_2^{1/2}\cdot \nabla (\gamma_2^{-1/2}v_1v_2)+\nabla v_1\cdot\nabla v_2]dx.
\end{equation}
Subtracting  \eqref{eq_2_10_13} from \eqref{eq_2_10_12} and using that $\tilde u_1=u_2$ on $\p \Omega$, and that $\Lambda_{\gamma_1}$ is symmetric, we obtain \eqref{eq_integral_identity}. The proof is complete. 
\end{proof}

\section{Proof of Theorem \ref{thm_main}} 
\label{sec_proof}

 Let us observe, as our starting point, that the fact that $\Lambda_{\gamma_1} f|_{\tilde F}=\Lambda_{\gamma_2}f|_{\tilde F}$ for any $f\in H^{1/2}(\p \Omega)$ together with the boundary determination result of \cite{Alessandrini_1990}, see also \cite{Kohn_Vogelius_1984}, and the assumptions of Theorem \ref{thm_main} implies that 
\begin{equation}
\label{eq_Cauchy_data_second_order}
\gamma_1=\gamma_2,  \quad \p_\nu\gamma_1=\p_\nu\gamma_2, \quad \text{on}\quad  \p \Omega.
\end{equation}

\subsection{Complex geometric optics solutions and interior estimates}
\label{subsection_interior}
Let $\Omega\subset \R^n$, $n\ge 3$, be a bounded open set  with $C^2$ boundary.  Let $\gamma_1,\gamma_2$ be such that either 
\begin{itemize}
\item[(i)] $\gamma_1,\gamma_2\in  C^{1,\delta}(\overline{\Omega})\cap H^{\frac{3}{2}}(\Omega)$, 
\end{itemize}
or
\begin{itemize}
\item[(ii)]  $\gamma_1,\gamma_2\in W^{1,\infty}(\Omega)\cap H^{\frac{3}{2}+\delta}(\Omega)$,  
\end{itemize}
where $0<\delta<1/2$ be arbitrarily small but fixed,
and $\gamma_1,\gamma_2>0$ on $\overline{\Omega}$.

As we saw in Subsection \ref{sec_CGO_H_3_2}, we can extend $\gamma_j$ to $\R^n$ so that either 
\begin{itemize}
\item[(i)] $\gamma_1,\gamma_2\in  C^{1,\delta}(\R^n)$,  $\gamma_j-1\in H^{\frac{3}{2}}(\R^n)$, 
\end{itemize}
or
\begin{itemize} 
\item[(ii)] $\gamma_1,\gamma_2\in W^{1,\infty}(\R^n)$, $\gamma_j-1\in H^{\frac{3}{2}+\delta}(\R^n)$, 
\end{itemize}
$\gamma_j>0$ on $\R^n$,  and  $\gamma_j=1$ near infinity.  

Using \eqref{eq_Cauchy_data_second_order},  we shall now modify $\gamma_2$ so that the extensions of $\gamma_1$ and $\gamma_2$  agree on $\R^n\setminus\Omega$, and their regularity is preserved. To that end, let 
\[
v=(\gamma_1-\gamma_2) 1_{\R^n\setminus \overline{\Omega}}\in \mathcal{E}'(\R^n),
\]
where $1_{\R^n\setminus\overline{ \Omega}}$ is the characteristic function of the set $\R^n\setminus \overline{\Omega}$.  In case (ii), using \eqref{eq_Cauchy_data_second_order} and the fact that $\frac{3}{2}+\delta$ is not a half-integer,  we conclude  from \cite[Theorem 5.1.14, Theorem 5.1.15]{Agranovich_book} that $v\in H^{\frac{3}{2}+\delta}(\R^n)$. It is also clear that  $v\in W^{1,\infty}(\R^n)$.

In case (i), the Sobolev index of $\gamma_j -1$ is a half-integer and therefore, we shall have to be a little more careful. We claim that $v\in C^{1,\delta}(\R^n)\cap H^{\frac{3}{2}}(\R^n)$, and when verifying this fact,  it suffices to work locally near a point at $\p \Omega$. Flattening out $\p \Omega$ by means of a $C^2$ diffeomorphism, we shall consider the regularity of $v=(\gamma_1-\gamma_2)1_{\overline{\R^n_-}}$, where $\gamma_1-\gamma_2\in C^{1,\delta}_0(\R^n)\cap H^{\frac{3}{2}}(\R^n)$, $\gamma_1=\gamma_2$, $\p_{x_n}\gamma_1=\p_{x_n}\gamma_2$ along $\{x_n=0\}$. It follows that $v\in C^{1,\delta}_0(\R^n)\cap H^{1}(\R^n)$, and it only remains to check that $\nabla v\in H^{\frac{1}{2}}(\R^n)$. To this end, we notice that by \cite[ formula (3.4.19)]{Agranovich_book}, this property is implied by the convergence of the integral, 
\begin{equation}
\label{eq_condition_Agranovich}
\int\!\!\!\int_{x_n<0} \frac{|\nabla(\gamma_1(x) - \gamma_2(x))|^2}{|x_n|}dx'dx_n<\infty.
\end{equation}
This is clear however, since $\nabla(\gamma_1-\gamma_2)\in C^{0,\delta}_0(\R^n)$ and $\nabla( \gamma_1-\gamma_2)=0$ along $x_n=0$. 

To achieve that $\gamma_1=\gamma_2$ on $\R^n\setminus\overline{\Omega}$ it now suffices to replace $\gamma_2$ by $\gamma_2+v$.

\textbf{Remark.} We would like to mention that we only need that $\nabla \gamma_j\in C^{0,\delta}(\overline{\Omega})$ in order to verify \eqref{eq_condition_Agranovich}, when making the conductivities equal on $\R^n\setminus\overline{\Omega}$, and preserving their regularity.  The rest of the argument in this work  in case (i) requires only $\gamma_j\in C^{1}(\overline{\Omega})\cap H^{\frac{3}{2}}(\Omega)$.

Let $\tilde \Omega\subset \R^n$ be an open set  such that $\Omega \subset\subset \tilde \Omega$, and let  $\varphi\in C^\infty(\tilde \Omega,\R)$ be a limiting Carleman weight for the semiclassical Laplacian on $\tilde \Omega$.  Let  $\hat \Omega\subset \R^n$ be  a bounded open set with $C^\infty$ boundary such that 
$\Omega\subset\subset \hat \Omega\subset\subset  \tilde \Omega$. 
Thanks to Proposition \ref{prop_CGO_solutions}, we know that there exist solutions $u_j\in H^1(\hat \Omega)$ of the equations $L_{\gamma_j}u_j=0$ in $\hat \Omega$ that are of the form 
\begin{equation}
\label{eq_2_10_15}
\begin{aligned}
u_1(x;h)&=\gamma_1^{-1/2}e^{-\frac{(\varphi+i\psi)}{h}}(a_1(x)+r_1(x;h)),\\
u_2(x;h)&=\gamma_2^{-1/2}e^{\frac{\varphi+i\psi}{h}}(a_2(x)+r_2(x;h))
\end{aligned}
\end{equation}
where  $\psi\in C^\infty(\tilde \Omega,\R)$ is a solution to the eikonal equation \eqref{eq_2_eikonal}, $a_j\in C^\infty(\overline{\hat \Omega})$ is a solution of the first transport equation \eqref{eq_2_10_3}, and the remainder term $r_j$ is such that 
\begin{equation}
\label{eq_2_10_16}
\|r_j\|_{L^2(\hat \Omega)}=o(h^{1/2}), \quad \|\nabla r_j\|_{L^2(\hat \Omega)}=o(h^{-1/2}),  \quad h\to 0. 
\end{equation}
Using the general estimate
\[
\| v|_{\p \Omega}\|_{H^{1/2}_{\text{scl}}(\p \Omega)}\le    \mathcal{O}(h^{-1/2})\|v\|_{H^1_{\text{scl}}(\Omega)},\quad v\in H^1(\Omega),
\]
see \cite{Sjostrand_2014},  we get
\begin{equation}
\label{b_terms_2_rem}
\|r_j\|_{L^2(\p \Omega)}=o(1).
\end{equation}

We shall substitute the solutions $u_1$ and $u_2$, given by \eqref{eq_2_10_15},  into the integral identity of Lemma \ref{lem_integral_identity}. Let us first understand the interior integral. 

\begin{lem}
\label{lem_full_data_lem}
Let $u_1$ and $u_2$ be given by \eqref{eq_2_10_15} and \eqref{eq_2_10_16}. Then 
\begin{align*}
\lim_{h\to 0} \int_{\Omega} \big[ -\nabla \gamma_1^{1/2} &\cdot \nabla (\gamma_2^{1/2}u_1 u_2)+\nabla \gamma_2^{1/2}\cdot \nabla(\gamma_1^{1/2}u_1u_2) \big]dx\\
&= \int_{\Omega}\big[- \nabla \gamma_1^{1/2}\cdot \nabla (\gamma_1^{-1/2} a_1a_2)+\nabla \gamma_2^{1/2}\cdot \nabla (\gamma_2^{-1/2}a_1a_2) \big]dx.
\end{align*}
 
\end{lem}

\begin{proof}
Using \eqref{eq_2_10_15}, we get 
\begin{equation}
\label{eq_2_10_17}
\begin{aligned}
\int_{\Omega} &\big[ -\nabla \gamma_1^{1/2} \cdot \nabla (\gamma_2^{1/2}u_1 u_2)+\nabla \gamma_2^{1/2}\cdot \nabla(\gamma_1^{1/2}u_1u_2) \big]dx\\
&=
\int_{\Omega}\big[- \nabla \gamma_1^{1/2}\cdot \nabla (\gamma_1^{-1/2} a_1a_2)+\nabla \gamma_2^{1/2}\cdot \nabla (\gamma_2^{-1/2}a_1a_2) \big]dx+ J_1+J_2,
\end{aligned}
\end{equation}
where
\begin{align*}
J_1&:=-\int_{\Omega}( \nabla \gamma_1^{1/2}\cdot \nabla \gamma_1^{-1/2} -\nabla \gamma_2^{1/2}\cdot \nabla \gamma_2^{-1/2}) (a_1r_2+ a_2r_1+r_1r_2) dx,\\
J_2&:=-\frac{1}{2}\int_{\Omega}(  \nabla \log \gamma_1 -\nabla \log \gamma_2) \cdot \nabla(a_1r_2+ a_2r_1+r_1r_2) dx.
\end{align*}

Let us show that $J_1, J_2\to 0$ as $h\to 0$. First using \eqref{eq_2_10_16}, we have
\begin{align*}
|J_1|&\le \|\nabla \gamma_1^{1/2}\cdot \nabla \gamma_1^{-1/2} -\nabla \gamma_2^{1/2}\cdot \nabla \gamma_2^{-1/2}\|_{L^\infty(\Omega)}\\
&(\|a_1\|_{L^2(\Omega)}\|r_2\|_{L^2(\Omega)}+\|a_2\|_{L^2(\Omega)}\|r_1\|_{L^2(\Omega)}+\|r_1\|_{L^2(\Omega)}\|r_2\|_{L^2(\Omega)})=o(h^{1/2}),
\end{align*}
as $ h\to 0$. 

To see that $J_2\to 0$ as $h\to 0$, we let $S:= \nabla \log \gamma_1 -\nabla \log \gamma_2\in (L^\infty\cap H^{1/2}\cap\mathcal{E}')(\R^n)$, and $S_h=S*\Psi_h$, where $\Psi_h(x)=h^{-n}\Psi(x/h)$, $h>0$, is the usual mollifier with $\Psi\in C^\infty_0(\R^n)$,  $0\le \Psi\le 1$, 
$\int \Psi dx=1$. 
We have 
\begin{align*}
|J_2|&\le \int_{\Omega} |S| |\nabla(r_1r_2)+r_2\nabla a_1 + r_1\nabla a_2|dx +  \int_{ \Omega} |S-S_h| \cdot |a_1\nabla r_2+ a_2\nabla r_1| dx\\
&+\bigg| \int_{ \Omega} S_h \cdot (a_1\nabla r_2+ a_2\nabla r_1) dx \bigg|:=J_{2,1}+J_{2,2}+J_{2,3}.
\end{align*} 

In view of \eqref{eq_2_10_16}, we see that 
\begin{align*}
J_{2,1}\le \|S\|_{L^\infty(\Omega)}(&\| \nabla r_1\|_{L^2(\Omega)}\|r_2\|_{L^2(\Omega)} + \|r_1\|_{L^2(\Omega)}\| \nabla  r_2\|_{L^2(\Omega)}\\
&+\|r_2\|_{L^2(\Omega)}\|\nabla a_1\|_{L^2(\Omega)}+ \|r_1\|_{L^2(\Omega)}\|\nabla a_2\|_{L^2(\Omega)})=o(1), \quad h\to 0.
\end{align*}
Using \eqref{eq_2_10_7} and \eqref{eq_2_10_16}, we get 
\begin{align*}
J_{2,2}\le \|S-S_h\|_{L^2(\Omega)}(\|a_1\|_{L^\infty}\|\nabla r_2\|_{L^2(\Omega)}+ \|a_2\|_{L^\infty}\|\nabla r_1\|_{L^2(\Omega)})=o(1), \quad h\to 0.
\end{align*}
To estimate $J_{2,3}$, it suffices to consider the integral 
\[
\int_{ \Omega} S_h\cdot a_1\nabla r_2dx= -\int_{\Omega}  r_2 \nabla\cdot (a_1S_h)dx +
\int_{\p \Omega}a_1r_2 S_h\cdot \nu dS.
\]
Using
\eqref{eq_2_10_8},  \eqref{eq_2_10_16} and \eqref{b_terms_2_rem}, we get 
\begin{align*}
\bigg| \int_{ \Omega} S_h\cdot a_1\nabla r_2dx\bigg|&\le \|r_2\|_{L^2(\Omega)} (\|a_1\|_{L^\infty} \|\nabla\cdot S_h\|_{L^2(\Omega)} +\|\nabla a_1\|_{L^\infty}\|S_h\|_{L^2(\Omega)})\\
&+\|a_1\|_{L^2(\p \Omega)}\|r_2\|_{L^2(\p \Omega)} \|S_h\|_{L^\infty(\p \Omega)}=o(1), \quad h\to 0.
\end{align*}
This shows that $J_1, J_2\to 0$ as $h\to 0$, and  completes the proof. 
\end{proof}

\begin{rem}
To establish Lemma \ref{lem_full_data_lem} we only need that $\gamma_j\in W^{1,\infty}(\Omega)\cap H^{3/2}(\Omega)$. 
\end{rem}

\subsection{Boundary term}
The purpose of this Subsection is to show that when substituting  the complex geometric optics solutions  $u_j\in H^1(\hat \Omega)$ of the equations $L_{\gamma_j}u_j=0$ in $\hat \Omega$, given by \eqref{eq_2_10_15}, into the integral identity of Lemma \ref{lem_integral_identity}, and letting $h\to 0$,  the boundary integral goes to zero.  To that end, notice that since  $u_j\in H^1(\hat \Omega)$ 
solves the equation
\begin{equation}
\label{eq_boundary_terms_eq}
-\Delta u_j-A_j\cdot \nabla u_j=0\quad \text{in}\quad \hat \Omega,
\end{equation}
with $A_j=\nabla \log \gamma_j\in (L^\infty\cap H^{1/2}\cap\mathcal{E}')(\R^n)$, $j=1,2$,  by interior elliptic regularity,  $u_j\in H^{2}(\Omega)$.

\begin{lem}
\label{lem_boundary_terms_int}
We have 
\begin{equation}
\label{eq_2_11_2}
J_b:=\int_{\p \Omega\setminus\tilde F} (\Lambda_{\gamma_1}\tilde u_1 -\Lambda_{\gamma_2}u_2)u_1dS\to 0, \quad h\to 0,
\end{equation}
where $u_1\in H^2(\Omega)$ and $u_2\in H^2(\Omega)$ are the complex geometric optics solutions, given by \eqref{eq_2_10_15},  to the equations $L_{\gamma_j}u_j=0$ in $\Omega$, $j=1,2$, respectively,  and  $\tilde u_1\in H^2(\Omega)$ is a solution to $L_{\gamma_1}\tilde u_1=0$ in $\Omega$ with $\tilde u_1=u_2$ on $\p \Omega$. 
\end{lem}

\begin{proof}
First since the traces $\p_\nu \tilde u_1|_{\p \Omega} \in H^{1/2}(\p \Omega)$ and $\p_\nu u_2|_{\p \Omega} \in H^{1/2}(\p \Omega)$ are well-defined and $\gamma_1=\gamma_2$ on $\p \Omega$, we have
\[
J_b=\int_{\p \Omega\setminus\tilde F} \gamma_1\p_\nu (\tilde u_1 -u_2)u_1dS.
\] 

Rather than working with the equations \eqref{eq_boundary_terms_eq}, we shall consider their conjugated versions. To that end,  let 
\[
w_j=\log \gamma_j, \quad w_{j,h}=w_j*\Psi_h\in C^\infty_0(\R^n), \quad A_{j,h}=\nabla w_{j,h}=A_j*\Psi_h\in C^\infty_0(\R^n),
\]
where $\Psi_h(x)$ is the mollifier given by \eqref{eq_mollifier_def} with a radial function $\Psi$. Consider the conjugated operators 
\[
e^{\frac{w_{j,h}}{2}}\circ (-\Delta -A_j\cdot \nabla) \circ e^{-\frac{w_{j,h}}{2}}=-\Delta + (A_{j,h}-A_j)\cdot \nabla +V_{j,h},
\]
where 
\[
V_{j,h}=\frac{\nabla \cdot A_{j,h}}{2}-\frac{(A_{j,h})^2}{4} +\frac{A_j\cdot A_{j,h}}{2}\in (L^\infty\cap H^{1/2}\cap\mathcal{E}')(\R^n).
\]

We have
\begin{equation}
\label{b_terms_0_eq_1}
(-\Delta + (A_{1,h}-A_1)\cdot \nabla +V_{1,h})(e^{\frac{w_{1,h}}{2}}\tilde u_1) =0 \quad \text{in}\quad \Omega,
\end{equation}
and 
\begin{equation}
\label{b_terms_0_eq_2}
(-\Delta + (A_{2,h}-A_2)\cdot \nabla +V_{2,h})(e^{\frac{w_{2,h}}{2}} u_2) =0 \quad \text{in}\quad \Omega.
\end{equation}

As a consequence of Lemma \ref{lem_Zhang_est}, we see that 
\begin{equation}
\label{eq_est_W_2}
\|A_{j,h}- A_j\|_{L^2(\R^n)}=o(h^{1/2}),
\end{equation}
\begin{equation}
\label{eq_est_W_infty}
\|A_{j,h}- A_j\|_{L^\infty(\R^n)}=\mathcal{O}(1),
\end{equation}
\begin{equation}
\label{eq_est_V_2}
\|V_{j,h}\|_{L^2(\R^n)}=o(h^{-1/2}),
\end{equation}
\begin{equation}
\label{eq_est_V_infty}
\|V_{j,h}\|_{L^\infty(\Omega)}=\mathcal{O}(h^{-1}).
\end{equation}

For future use, we remark that if $\gamma_j\in C^1(\R^n)$, then $w_j\in C^1_0(\R^n)$, and if $\gamma_j\in W^{1,\infty}(\R^n)$ is such that $\gamma_j-1\in H^{\frac{3}{2}+\delta}(\R^n)\cap\mathcal{E}'(\R^n)$, then $w_j\in H^{\frac{3}{2}+\delta}(\R^n)$. To see the latter fact,  we first observe that $w_j\in (W^{1,\infty}\cap \mathcal{E}')(\R^n)\subset L^2(\R^n)$. We also have
\[
\Delta w_j=\div(\gamma_j^{-1}\nabla(\gamma_j-1))\in H^{-\frac{1}{2}+\delta}(\R^n),
\]
since $\gamma_j^{-1}\nabla(\gamma_j-1)\in H^{\frac{1}{2}+\delta}(\R^n)$ as $\delta>0$ is small. Thus, $(1-\Delta) w_j\in H^{-\frac{1}{2}+\delta}(\R^n)$, and by global elliptic regularity, we conclude that $w_j\in H^{\frac{3}{2}+\delta}(\R^n)$. We shall therefore be able to apply Lemma \ref{lem_app_delta} to $w_j$.

Now since $\tilde u_1=u_2$ on $\p \Omega$, we have  
\begin{align*}
\p_\nu (\tilde u_1 &-u_2)=e^{-\frac{w_{1,h}}{2}}\p_\nu (e^{\frac{w_{1,h}}{2}}(\tilde u_1-u_2))\\
&=e^{-\frac{w_{1,h}}{2}}
\p_\nu (e^{\frac{w_{1,h}}{2}}\tilde u_1- e^{\frac{w_{2,h}}{2}}u_2)+ e^{-\frac{w_{1,h}}{2}}
\p_\nu ((e^{\frac{w_{2,h}}{2}}- e^{\frac{w_{1,h}}{2}})u_2)\quad \text{on}\quad \p \Omega. 
\end{align*}
Thus, we get 
\begin{equation}
\label{b_terms_def_J_sum}
J_b=J_{b,1}+J_{b,2},
\end{equation}
where 
\begin{equation}
\label{b_terms_def_J_b_1}
J_{b,1}=\int_{\p \Omega\setminus\tilde F} \gamma_1  e^{-\frac{w_{1,h}}{2}}
\p_\nu (e^{\frac{w_{1,h}}{2}}\tilde u_1- e^{\frac{w_{2,h}}{2}}u_2) u_1dS,
\end{equation}
\begin{equation}
\label{b_terms_def_J_b_2}
J_{b,2}=\int_{\p \Omega\setminus\tilde F} \gamma_1  e^{-\frac{w_{1,h}}{2}}
\p_\nu ((e^{\frac{w_{2,h}}{2}}- e^{\frac{w_{1,h}}{2}})u_2) u_1dS.
\end{equation}

Let us start by estimating $J_{b,1}$. In doing so we shall use the boundary Carleman estimates of Proposition \ref{prop_boundary_Carleman} applied to the operator in \eqref{b_terms_0_eq_1}.  First notice that when $\varphi$ is given by \eqref{eq_varphi}, we have $\p_\nu\varphi(x)=\frac{(x-x_0)\cdot\nu(x)}{|x-x_0|^2}$, and therefore, $\p \Omega_-=F(x_0)$.  By the definition of $F(x_0)$ and $\tilde F$, there exists $\varepsilon>0$ such that 
\begin{equation}
\label{eq_definition_esp}
\p \Omega_{-}=F(x_0)\subset F_{\varepsilon}:=\{x\in \p \Omega:\p_\nu \varphi(x)\le \varepsilon\}\subset \tilde F. 
\end{equation}

Substituting $u_1$ given by \eqref{eq_2_10_15} into $J_{b,1}$, using \eqref{eq_cond_Lipschitz_nabla} and the Cauchy--Schwarz inequality,  we get 
\begin{equation}
\label{b_terms_2}
\begin{aligned}
|J_{b,1}|\le \mathcal{O}(1) \bigg(\int_{\p \Omega\setminus F_{\varepsilon}} \frac{1}{\varepsilon}\varepsilon|\p_\nu (e^{\frac{w_{1,h}}{2}}\tilde u_1- e^{\frac{w_{2,h}}{2}}u_2) |^2 e^{-\frac{2\varphi}{h}}dS\bigg)^{1/2}\|a_1+r_1\|_{L^2(\p \Omega)}.
\end{aligned}
\end{equation}
It follows from \eqref{b_terms_2} and \eqref{b_terms_2_rem} that 
\begin{equation}
\label{b_terms_3_1}
|J_{b,1}|\le \mathcal{O}(1)\bigg(\int_{\p \Omega_+} (\p_\nu\varphi)  e^{-\frac{2\varphi}{h}} |\p_\nu (e^{\frac{w_{1,h}}{2}}\tilde u_1- e^{\frac{w_{2,h}}{2}}u_2) |^2 dS\bigg)^{1/2}.
\end{equation}

Using the boundary Carleman estimate \eqref{eq_2_11_7} for the operator $-h^2\Delta+ h(A_{1,h}-A_1)\cdot h\nabla +h^2V_{1,h}$, we get 
\begin{equation}
\label{b_terms_4}
\begin{aligned}
&\int_{\p \Omega_+} (\p_\nu\varphi)  e^{-\frac{2\varphi}{h}} |\p_\nu (e^{\frac{w_{1,h}}{2}}\tilde u_1- e^{\frac{w_{2,h}}{2}}u_2) |^2 dS\\
&\le \mathcal{O}(h) \| e^{-\varphi/h} (-\Delta + (A_{1,h}-A_1)\cdot \nabla +V_{1,h}) (e^{\frac{w_{1,h}}{2}}\tilde u_1- e^{\frac{w_{2,h}}{2}}u_2)\|^2_{L^2(\Omega)}\\
&+\mathcal{O}(h^{-2})\|e^{-\frac{\varphi}{h}}(e^{\frac{w_{1,h}}{2}}\tilde u_1- e^{\frac{w_{2,h}}{2}}u_2)\|^2_{L^2(\p \Omega)}\\
&+\mathcal{O}(h^{-1})\int_{\p \Omega} e^{-\frac{2\varphi}{h}}|\p_\nu (e^{\frac{w_{1,h}}{2}}\tilde u_1- e^{\frac{w_{2,h}}{2}}u_2)||e^{\frac{w_{1,h}}{2}}\tilde u_1- e^{\frac{w_{2,h}}{2}}u_2|dS \\
&+\mathcal{O}(1)\int_{\p \Omega_{-}}(-\p_\nu \varphi)e^{-\frac{2\varphi}{h}}|\p_\nu (e^{\frac{w_{1,h}}{2}}\tilde u_1- e^{\frac{w_{2,h}}{2}}u_2)|^2dS\\
&+\mathcal{O}(1)\|e^{-\frac{\varphi}{h}}\nabla_t (e^{\frac{w_{1,h}}{2}}\tilde u_1- e^{\frac{w_{2,h}}{2}}u_2)\|^2_{L^2(\p \Omega)} \\
&+\mathcal{O}(1)\int_{\p \Omega} e^{-\frac{2\varphi}{h}}|\nabla_t (e^{\frac{w_{1,h}}{2}}\tilde u_1- e^{\frac{w_{2,h}}{2}}u_2)||\p_\nu (e^{\frac{w_{1,h}}{2}}\tilde u_1- e^{\frac{w_{2,h}}{2}}u_2) |dS.
\end{aligned}
\end{equation}

Let us proceed by estimating each term in the right hand side of \eqref{b_terms_4}. Using \eqref{b_terms_0_eq_1} and \eqref{b_terms_0_eq_2},  we obtain that 
\begin{equation}
\label{b_terms_5}
\begin{aligned}
\mathcal{O}(h) \| e^{-\varphi/h} (-\Delta + (A_{1,h}-A_1)\cdot \nabla +V_{1,h}) (e^{\frac{w_{1,h}}{2}}\tilde u_1- e^{\frac{w_{2,h}}{2}}u_2)\|^2_{L^2(\Omega)}\\
\le
\mathcal{O}(h) \| e^{-\varphi/h} \big((A_{1,h}-A_1)-(A_{2,h}-A_2)\big)\cdot \nabla (e^{\frac{w_{2,h}}{2}}u_2)\|^2_{L^2(\Omega)}\\
+\mathcal{O}(h)\| e^{-\varphi/h} (V_{1,h}-V_{2,h}) (e^{\frac{w_{2,h}}{2}}u_2)\|^2_{L^2(\Omega)}.
\end{aligned}
\end{equation}

Substituting $u_2$ given by \eqref{eq_2_10_15}, and using  \eqref{eq_cond_Lipschitz_nabla}, \eqref{eq_2_10_16}, \eqref{eq_est_V_infty}, and \eqref{eq_est_V_2}, we get 
\begin{equation}
\label{b_terms_6}
\begin{aligned}
&\mathcal{O}(h^{1/2})\| e^{-\varphi/h} (V_{1,h}-V_{2,h}) (e^{\frac{w_{2,h}}{2}}u_2)\|_{L^2(\Omega)}\\
&\le 
\mathcal{O}(h^{1/2})(\|V_{1,h}-V_{2,h}\|_{L^2(\R^n)}\|a_2\|_{L^\infty(\Omega)}+ \|V_{1,h}-V_{2,h}\|_{L^\infty(\R^n)}\|r_2\|_{L^2(\Omega)})=o(1),
\end{aligned}
\end{equation}
as $h\to 0$.

Using \eqref{eq_2_10_15}, \eqref{eq_2_10_16}, \eqref{eq_cond_Lipschitz_nabla}, \eqref{eq_est_W_infty}, we obtain that 
\begin{equation}
\label{b_terms_7}
\begin{aligned}
\mathcal{O}(h^{1/2})& \| e^{-\varphi/h} \big((A_{1,h}-A_1)-(A_{2,h}-A_2)\big)\cdot \nabla (e^{\frac{w_{2,h}}{2}}) u_2\|_{L^2(\Omega)}\\
&\le \mathcal{O}(h^{1/2}) \|A_{1,h}-A_1-A_{2,h}+A_2 \|_{L^\infty(\R^n)} \|\nabla (e^{\frac{w_{2,h}}{2}})\|_{L^\infty(\R^n)}\\
&\|\gamma_2^{-1/2}\|_{L^\infty(\R^n)} \|a_2+r_2\|_{L^2(\Omega)}\le \mathcal{O}(h^{1/2}),
\end{aligned}
\end{equation}
as $h\to 0$. 

Using \eqref{eq_2_10_15}, \eqref{eq_2_10_16}, \eqref{eq_cond_Lipschitz_nabla}, \eqref{eq_est_W_2}, \eqref{eq_est_W_infty}, we get
\begin{equation}
\label{b_terms_8}
\begin{aligned}
\mathcal{O}&(h^{1/2}) \| e^{-\varphi/h} \big((A_{1,h}-A_1)-(A_{2,h}-A_2)\big) e^{\frac{w_{2,h}}{2}}\cdot \nabla u_2\|_{L^2(\Omega)}\\
&\le \mathcal{O}(h^{1/2}) \|  \big((A_{1,h}-A_1)-(A_{2,h}-A_2)\big) e^{\frac{w_{2,h}}{2}} \nabla(\gamma_2^{-1/2}) (a_2+r_2)\|_{L^2(\Omega)}\\
&+\mathcal{O}(h^{1/2}) \|  \big((A_{1,h}-A_1)-(A_{2,h}-A_2)\big) e^{\frac{w_{2,h}}{2}} \gamma_2^{-1/2}\frac{\nabla \varphi+i\nabla \psi}{h} (a_2+r_2)\|_{L^2(\Omega)}\\
&+\mathcal{O}(h^{1/2}) \|  \big((A_{1,h}-A_1)-(A_{2,h}-A_2)\big) e^{\frac{w_{2,h}}{2}} \gamma_2^{-1/2} (\nabla a_2+\nabla r_2)\|_{L^2(\Omega)}\\
&\le \mathcal{O}(h^{-1/2}) \|  A_{1,h}-A_1-(A_{2,h}-A_2)\|_{L^2(\R^n)} \|a_2\|_{L^\infty(\Omega)} \\
&+ \mathcal{O}(h^{-1/2}) \|  A_{1,h}-A_1-(A_{2,h}-A_2)\|_{L^\infty(\R^n)} \|r_2\|_{L^2(\Omega)}\\
&+\mathcal{O}(h^{1/2}) \|  A_{1,h}-A_1-(A_{2,h}-A_2)\|_{L^\infty(\R^n)} \|\nabla a_2 +\nabla r_2\|_{L^2(\Omega)}=o(1),
\end{aligned}
\end{equation}
as $h\to 0$. 

Combining \eqref{b_terms_5}, \eqref{b_terms_6}, \eqref{b_terms_7},  and \eqref{b_terms_8}, we conclude that 
\begin{equation}
\label{b_terms_9}
\mathcal{O}(h^{1/2}) \| e^{-\varphi/h} (-\Delta + (A_{1,h}-A_1)\cdot \nabla +V_{1,h}) (e^{\frac{w_{1,h}}{2}}\tilde u_1- e^{\frac{w_{2,h}}{2}}u_2)\|_{L^2(\Omega)}=o(1),
\end{equation}
as $h\to 0$.

Let us now estimate the second term in the right hand side of \eqref{b_terms_4}.  
Using the equalities  $\tilde u_1=u_2$ on $\p \Omega$, $\gamma_1=\gamma_2$ on $\p \Omega$, 
 \eqref{eq_2_10_15}, and the estimate
\begin{equation}
\label{b_terms_10}
|e^z-e^w|\le |z-w|e^{\max\{\text{Re} z, \text{Re} w\}},\quad z,w\in\C,
\end{equation}
  we get 
\begin{equation}
\label{b_terms_10_0}
\begin{aligned}
\mathcal{O}(h^{-2})&\|e^{-\frac{\varphi}{h}}(e^{\frac{w_{1,h}}{2}}\tilde u_1- e^{\frac{w_{2,h}}{2}}u_2)\|^2_{L^2(\p \Omega)}=\mathcal{O}(h^{-2})\|e^{-\frac{\varphi}{h}}(e^{\frac{w_{1,h}}{2}}- e^{\frac{w_{2,h}}{2}})u_2\|^2_{L^2(\p \Omega)}\\
&\le \mathcal{O}(h^{-2}) \int_{\p \Omega} |w_{1,h}- w_{2,h}|^2(|a_2|^2 +|r_2|^2)dS\\
&\le \mathcal{O}(h^{-2}) \|a_2\|^2_{L^\infty(\p\Omega)}\big(\| w_{1,h}-w_1\|_{L^2(\p \Omega)}^2+ \| w_{2,h}-w_2\|_{L^2(\p \Omega)}^2 \big)\\
&+\mathcal{O}(h^{-2}) (\| w_{1,h}-w_1\|^2_{L^\infty(\R^n)}+\| w_{2,h}-w_2\|^2_{L^\infty(\R^n)})\|r_2\|_{L^2(\p \Omega)}^2=o(1),
\end{aligned}
\end{equation}
as $h\to 0$. 
Here we have also  used \eqref{b_terms_2_rem}, \eqref{eq_cond_Lipschitz}, and \eqref{eq_cond_delta_1}.

Let us now estimate the fourth term in the right hand side of \eqref{b_terms_4}. When doing so, it is convenient to write 
\begin{equation}
\label{b_terms_10_1}
\begin{aligned}
\int_{\p \Omega_{-}}(-\p_\nu \varphi)e^{-\frac{2\varphi}{h}}&|\p_\nu (e^{\frac{w_{1,h}}{2}}\tilde u_1- e^{\frac{w_{2,h}}{2}}u_2)|^2dS\\
&\le \mathcal{O}(1) \int_{F_\varepsilon}e^{-\frac{2\varphi}{h}}|\p_\nu (e^{\frac{w_{1,h}}{2}}\tilde u_1- e^{\frac{w_{2,h}}{2}}u_2)|^2dS
\end{aligned}
\end{equation}
and to estimate the latter integral as we will need it later. Since $\tilde u_1=u_2$ on $\p \Omega$, and $\gamma_1=\gamma_2$ on $\p \Omega$, we have 
\[
\p_\nu \tilde u_1|_{F_\varepsilon}=\p_\nu u_2|_{F_\varepsilon},
\]
and therefore, 
\begin{equation}
\label{b_terms_11}
\p_\nu (e^{\frac{w_{1,h}}{2}}\tilde u_1- e^{\frac{w_{2,h}}{2}}u_2)=u_2\p_\nu (e^{\frac{w_{1,h}}{2}}- e^{\frac{w_{2,h}}{2}})+ (e^{\frac{w_{1,h}}{2}}- e^{\frac{w_{2,h}}{2}})\p_\nu u_2\quad \text{on}\quad F_\varepsilon.
\end{equation}

Hence, using \eqref{b_terms_11} and \eqref{b_terms_10}, we get
\begin{equation}
\label{b_terms_11_2}
\begin{aligned}
\int_{F_\varepsilon}e^{-\frac{2\varphi}{h}}|\p_\nu (e^{\frac{w_{1,h}}{2}}\tilde u_1- e^{\frac{w_{2,h}}{2}}u_2)|^2dS\le 
\mathcal{O}(1) \bigg( \int_{F_\varepsilon}e^{-\frac{2\varphi}{h}}|\p_\nu (e^{\frac{w_{1,h}}{2}}- e^{\frac{w_{2,h}}{2}})|^2|u_2|^2dS\\
+ \int_{F_\varepsilon}e^{-\frac{2\varphi}{h}}|w_{1,h}- w_{2,h}|^2|\p_\nu u_2|^2dS\bigg).
\end{aligned}
\end{equation}

Let us now estimate the first term in the right hand side of \eqref{b_terms_11_2}. Using the fact that   $\gamma_1=\gamma_2$, and $\p_\nu\gamma_1=\p_\nu\gamma_2 $ on $\p \Omega$, we  
have on $\p \Omega$, 
\begin{equation}
\label{b_terms_11_3}
\begin{aligned}
|\p_\nu (e^{\frac{w_{1,h}}{2}}- e^{\frac{w_{2,h}}{2}})|\le e^{\frac{w_{1,h}}{2}} |\p_\nu w_{1,h}-\p_\nu w_{2,h}|+\mathcal{O}(1)|w_{1,h}-w_{2,h}| |\p_\nu w_{2,h}|\\
\le \mathcal{O}(1) ( |\p_\nu w_{1,h}-\p_\nu w_1 | +|\p_\nu w_{2,h}-\p_\nu w_2| +|w_{1,h}-w_{1}| + |w_{2,h}-w_{2}| ).
\end{aligned}
\end{equation}
Therefore, by \eqref{eq_cond_delta_1}, \eqref{eq_cond_delta_2} and  \eqref{eq_cond_Lipschitz_nabla}, we get  from the second line in \eqref{b_terms_11_3}, 
\begin{equation}
\label{b_terms_12}
\|\p_\nu (e^{\frac{w_{1,h}}{2}}- e^{\frac{w_{2,h}}{2}})\|_{L^2(\p \Omega)}=o(1),
\end{equation}
and from the first line in \eqref{b_terms_11_3},
\begin{equation}
\label{b_terms_13}
\|\p_\nu (e^{\frac{w_{1,h}}{2}}- e^{\frac{w_{2,h}}{2}})\|_{L^\infty(\p \Omega)}=O(1),
\end{equation}
as $h\to 0$. 

Hence, using \eqref{b_terms_2_rem}, \eqref{b_terms_12} and \eqref{b_terms_13}, we obtain that 
\begin{equation}
\label{b_terms_13_1}
\begin{aligned}
  &\int_{F_\varepsilon}e^{-\frac{2\varphi}{h}}|\p_\nu (e^{\frac{w_{1,h}}{2}}- e^{\frac{w_{2,h}}{2}})|^2|u_2|^2dS \le  2 \int_{\p\Omega}|\p_\nu (e^{\frac{w_{1,h}}{2}}- e^{\frac{w_{2,h}}{2}})|^2(|a_2|^2+|r_2|^2)dS\\
 & \le 2 \|\p_\nu (e^{\frac{w_{1,h}}{2}}- e^{\frac{w_{2,h}}{2}})\|_{L^2(\p \Omega)}^2   \|a_2\|^2_{L^\infty(\p \Omega)}  + 2 \|\p_\nu (e^{\frac{w_{1,h}}{2}}- e^{\frac{w_{2,h}}{2}})\|_{L^\infty(\p \Omega)}^2  \|r_2\|^2_{L^2(\p \Omega)}\\
 &=o(1), \quad h\to 0.
\end{aligned}
\end{equation}

In order to estimate the second term in the right hand side of \eqref{b_terms_11_2}, we write 
\begin{equation}
\label{b_terms_13_1_1}
u_2=e^{\frac{\varphi+i\psi}{h}}v_2, \quad v_2=\gamma_2^{-1/2}(a_2+r_2)\in H^{2}(\Omega),
\end{equation}
and since we do not have an estimate for $\p_\nu r_2|_{\p \Omega}$, we shall proceed as follows. 
First,
\begin{equation}
\label{b_terms_13_2}
|\p_\nu u_2|^2= \bigg|e^{\frac{\varphi+i\psi}{h}}\bigg( \frac{\p_\nu \varphi+i\p_\nu \psi}{h} v_2 + \p_\nu v_2\bigg)\bigg|^2\le \mathcal{O}(1)e^{\frac{2\varphi}{h}}\bigg(\frac{1}{h^2}|v_2|^2+|\p_\nu v_2|^2\bigg)\ \text{on}\ \p \Omega.
\end{equation}
Thus, 
\begin{equation}
\label{b_terms_14}
\begin{aligned}
\int_{F_\varepsilon}e^{-\frac{2\varphi}{h}}|w_{1,h}- w_{2,h}|^2|\p_\nu u_2|^2dS\le &\mathcal{O}(h^{-2}) \int_{\p \Omega}|w_{1,h}- w_{2,h}|^2|v_2|^2dS\\
&+ \mathcal{O}(1)\int_{\p \Omega}|w_{1,h}- w_{2,h}|^2|\p_\nu v_2|^2dS. 
\end{aligned}
\end{equation}

For the first term in the right hand side of \eqref{b_terms_14}, using that $\gamma_1=\gamma_2$ on $\p \Omega$ and \eqref{b_terms_2_rem}, \eqref {eq_cond_delta_1} and \eqref{eq_cond_Lipschitz}, we get 
\begin{equation}
\label{b_terms_15}
\begin{aligned}
\mathcal{O}(h^{-2}) &\int_{\p \Omega}|w_{1,h}- w_{2,h}|^2|v_2|^2dS\\
&\le 
\mathcal{O}(h^{-2}) \int_{\p \Omega}(|w_{1,h} -w_1|^2+|w_{2,h}-w_2|^2 ) (|a_2|^2+|r_2|^2)dS\\
&\le \mathcal{O}(h^{-2})(\|w_{1,h} -w_1\|^2_{L^2(\p \Omega)}+\|w_{2,h}-w_2\|^2_{L^2(\p\Omega)})\|a_2\|_{L^\infty(\p \Omega)}^2\\
&+\mathcal{O}(h^{-2})(\|w_{1,h} -w_1\|^2_{L^\infty(\p \Omega)}+\|w_{2,h}-w_2\|^2_{L^\infty(\p\Omega)})\|r_2\|_{L^2(\p \Omega)}^2=o(1),
\end{aligned}
\end{equation}
as $h\to 0$.  

In order to estimate the second term in the right hand side of \eqref{b_terms_14},  we shall use the following result of \cite[Lemma 2.2]{Zhang_Guo_2012}: let $u\in H^1(\Omega)$, then 
\[
\|u\|^2_{L^2(\p \Omega)}\le C(\|u\|_{L^2(\Omega)}\|\nabla u \|_{L^2(\Omega)}+\|u\|^2_{L^2(\Omega)}),
\]
where the constant $C>0$ depends only on $\Omega$ and $n$. Using this estimate together with  the interior elliptic regularity for the Laplacian, we get 
\begin{equation}
\label{b_terms_15_1}
\|\nabla v_2\|_{L^2(\p \Omega)} \le C (\|\nabla v_2\|_{L^2( \Omega)}\|\Delta v_2\|_{L^2(\hat \Omega)}+\|\nabla v_2\|_{L^2( \hat \Omega)}^2)^{1/2}.
\end{equation}
It follows from \eqref{eq_2_10_16} that 
\begin{equation}
\label{b_terms_16}
\|\nabla v_2\|_{L^2(\hat \Omega)}=o(h^{-1/2}), \quad h\to 0.
\end{equation}
Since $u_2$ solves the equation $-\Delta u_2-\nabla\log \gamma_2\cdot \nabla u_2=0$ in $\hat \Omega$, $v_2$ solves the equation
\begin{equation}
\label{b_terms_17}
-\Delta v_2=\bigg(2\frac{\nabla \varphi+i\nabla\psi}{h}+\nabla\log\gamma_2\bigg)\cdot \nabla v_2+
\bigg(\frac{\Delta \varphi+i\Delta\varphi}{h}+\nabla \log \gamma_2 \cdot\frac{\nabla\varphi+i\nabla \varphi}{h}\bigg)v_2
\end{equation}
in $\hat \Omega$, thanks to \eqref{eq_2_eikonal}. Hence, using that $\|v_2\|_{L^2(\hat \Omega)}=\mathcal{O}(1)$, \eqref{b_terms_16}, and \eqref{b_terms_17}, we get 
\begin{equation}
\label{b_terms_18_old}
\|\Delta v_2\|_{L^2(\hat \Omega)}=o(h^{-3/2}), \quad h\to 0. 
\end{equation}

It follows from \eqref{b_terms_15_1}, \eqref{b_terms_16} and \eqref{b_terms_18_old}
\begin{equation}
\label{b_terms_18}
\|\nabla v_2\|_{L^2(\p\Omega)}=o(h^{-1}), \quad h\to 0. 
\end{equation}

For the second term in the right hand side of \eqref{b_terms_14}, using \eqref{b_terms_18} and \eqref{eq_cond_Lipschitz}, we get 
\begin{equation}
\label{b_terms_19}
\begin{aligned}
\int_{\p \Omega}|w_{1,h}- w_{2,h}|^2|\p_\nu v_2|^2dS \le 2(\|w_{1,h}-w_1\|_{L^\infty(\p \Omega)}^2 &+\|w_{2,h}-w_2\|_{L^\infty(\p \Omega)}^2)\\
&\|\p_\nu v_2\|^2_{L^2(\p \Omega)}=o(1),\quad h\to 0.
\end{aligned}
\end{equation}

Thus, we conclude from \eqref{b_terms_14}, \eqref{b_terms_15} and \eqref{b_terms_19} that 
\begin{equation}
\label{b_terms_20}
\int_{F_\varepsilon}e^{-\frac{2\varphi}{h}}|w_{1,h}- w_{2,h}|^2|\p_\nu u_2|^2dS=o(1),\quad h\to 0.
\end{equation}
It follows from \eqref{b_terms_11_2}, \eqref{b_terms_13_1} and \eqref{b_terms_20} that 
\begin{equation}
\label{b_terms_21}
\int_{F_\varepsilon}e^{-\frac{2\varphi}{h}}|\p_\nu (e^{\frac{w_{1,h}}{2}}\tilde u_1- e^{\frac{w_{2,h}}{2}}u_2)|^2dS=o(1),\quad h\to 0,
\end{equation}
and therefore, in view of \eqref{b_terms_10_1}
\begin{equation}
\label{b_terms_22}
\int_{\p \Omega_{-}}(-\p_\nu \varphi)e^{-\frac{2\varphi}{h}}|\p_\nu (e^{\frac{w_{1,h}}{2}}\tilde u_1- e^{\frac{w_{2,h}}{2}}u_2)|^2dS=o(1),\quad h\to 0.
\end{equation}

Let us now estimate the fifth term in the right hand side of \eqref{b_terms_4}. First as $\tilde u_1=u_2$ on $\p \Omega$, we have
\begin{equation}
\label{b_terms_22_1}
\nabla_t (e^{\frac{w_{1,h}}{2}}\tilde u_1- e^{\frac{w_{2,h}}{2}}u_2)=u_2\nabla_t(e^{\frac{w_{1,h}}{2}}- e^{\frac{w_{2,h}}{2}})+(e^{\frac{w_{1,h}}{2}}- e^{\frac{w_{2,h}}{2}})\nabla_t u_2. 
\end{equation}
Similarly to \eqref{b_terms_11_3}, \eqref{b_terms_12} and \eqref{b_terms_13}, we get
\begin{equation}
\label{b_terms_23}
\|\nabla_t (e^{\frac{w_{1,h}}{2}}- e^{\frac{w_{2,h}}{2}})\|_{L^2(\p \Omega)}=o(1),
\end{equation}
\begin{equation}
\label{b_terms_24}
\|\nabla_t (e^{\frac{w_{1,h}}{2}}- e^{\frac{w_{2,h}}{2}})\|_{L^\infty(\p \Omega)}=\mathcal{O}(1),
\end{equation}
as $h\to 0$. Using \eqref{b_terms_22_1} together with \eqref{b_terms_13_2}, we get 
\begin{equation}
\label{b_terms_25}
\begin{aligned}
\|e^{-\frac{\varphi}{h}}\nabla_t &(e^{\frac{w_{1,h}}{2}}\tilde u_1- e^{\frac{w_{2,h}}{2}}u_2)\|^2_{L^2(\p \Omega)}\\
&\le \mathcal{O}(1) \int_{\p \Omega} |\nabla_t(e^{\frac{w_{1,h}}{2}}- e^{\frac{w_{2,h}}{2}})|^2(|a_2|^2+|r_2|^2)dS\\
&+\mathcal{O}(h^{-2})\int_{\p\Omega}|w_{1,h}-w_{2,h}|^2|v_2|^2dS+\mathcal{O}(1)\int_{\p \Omega}|w_{1,h}-w_{2,h}|^2|\nabla_t v_2|^2dS\\
&=o(1), \quad h\to 0, 
\end{aligned}
\end{equation}
where the latter estimate is established as in \eqref{b_terms_13_1},  \eqref{b_terms_15},  \eqref{b_terms_19} with the help of \eqref{b_terms_18}.

Let us now estimate the third term in the right hand side of \eqref{b_terms_4}. Letting 
\[
v=e^{\frac{w_{1,h}}{2}}\tilde u_1- e^{\frac{w_{2,h}}{2}}u_2,
\]
recalling the fixed positive number $\varepsilon$ defined in \eqref{eq_definition_esp}, and 
using  the Cauchy--Schwarz  and  Peter--Paul inequalities,  we get   
\begin{equation}
\label{b_terms_26}
\begin{aligned}
\mathcal{O}&(h^{-1})\int_{\p \Omega} e^{-\frac{2\varphi}{h}}|\p_\nu (e^{\frac{w_{1,h}}{2}}\tilde u_1- e^{\frac{w_{2,h}}{2}}u_2)||e^{\frac{w_{1,h}}{2}}\tilde u_1- e^{\frac{w_{2,h}}{2}}u_2|dS\\
&\le \mathcal{O}(h^{-1})\|e^{-\frac{\varphi}{h}}\p_\nu v\|_{L^2(\p \Omega)}\|e^{-\frac{\varphi}{h}} v\|_{L^2(\p \Omega)}\\
&\le 
\frac{\varepsilon}{4}\|e^{-\frac{\varphi}{h}}\p_\nu v\|_{L^2(\p \Omega)}^2+\mathcal{O}( h^{-2})\|e^{-\frac{\varphi}{h}} v\|_{L^2(\p \Omega)}^2 \le   \mathcal{O}(h^{-2})\|e^{-\frac{\varphi}{h}} v\|_{L^2(\p \Omega)}^2\\
&+\frac{\varepsilon}{4}\int_{F_\varepsilon}e^{-\frac{2\varphi}{h}}|\p_\nu v|^2 dS +
\frac{1}{4}\int_{\p \Omega\setminus F_\varepsilon}(\p_\nu \varphi)e^{-\frac{2\varphi}{h}}|\p_\nu v|^2 dS\\
&\le o(1)+\frac{1}{4} \int_{\p \Omega_{+}}(\p_\nu \varphi)e^{-\frac{2\varphi}{h}}|\p_\nu v|^2 dS,
\end{aligned}
\end{equation}
as $h\to 0$. Here we have also used \eqref{b_terms_10_0} and \eqref{b_terms_21}.

To estimate the final sixth term in the right hand side of \eqref{b_terms_4}, we proceed similarly to \eqref{b_terms_26} and obtain that 
\begin{equation}
\label{b_terms_27}
\begin{aligned}
\int_{\p \Omega} e^{-\frac{2\varphi}{h}}|\nabla_t (e^{\frac{w_{1,h}}{2}}\tilde u_1&- e^{\frac{w_{2,h}}{2}}u_2)||\p_\nu (e^{\frac{w_{1,h}}{2}}\tilde u_1- e^{\frac{w_{2,h}}{2}}u_2) |dS\\
&\le 
\frac{\varepsilon}{4} \|e^{-\frac{\varphi}{h}}\p_\nu v\|^2_{L^2(\p \Omega)}+
\mathcal{O}(1) \|e^{-\frac{\varphi}{h}}\nabla_t v\|^2_{L^2(\p \Omega)}\\
&\le o(1)+\frac{1}{4} \int_{\p \Omega_{+}}(\p_\nu \varphi)e^{-\frac{2\varphi}{h}}|\p_\nu v|^2 dS,
\end{aligned}
\end{equation}
as $h\to 0$. Here we have also used \eqref{b_terms_25} and \eqref{b_terms_21}.

Combining \eqref{b_terms_4}, \eqref{b_terms_9},  \eqref{b_terms_10_0}, \eqref{b_terms_22}, \eqref{b_terms_25}, \eqref{b_terms_26}, and \eqref{b_terms_27}, we get
\[
\int_{\p \Omega_+} (\p_\nu\varphi)  e^{-\frac{2\varphi}{h}} |\p_\nu (e^{\frac{w_{1,h}}{2}}\tilde u_1- e^{\frac{w_{2,h}}{2}}u_2) |^2 dS=o(1), \quad h\to 0.
\]
Hence, in view of \eqref{b_terms_3_1}, 
\begin{equation}
\label{b_terms_28}
|J_{b,1}|=o(1), \quad h\to 0, 
\end{equation}
where $J_{b,1}$ is given by \eqref{b_terms_def_J_b_1}. 

Let us finally show that 
\begin{equation}
\label{b_terms_28_1}
|J_{b,2}|=o(1), \quad h\to 0,
\end{equation}
 where $J_{b,2}$ is defined by \eqref{b_terms_def_J_b_2}.  We have
\begin{equation}
\label{b_terms_29}
|J_{b,2}|\le \mathcal{O}(1) \int_{\p \Omega} |\p_\nu (e^{\frac{w_{2,h}}{2}}- e^{\frac{w_{1,h}}{2}})| |u_2| |u_1|dS+  \mathcal{O}(1)
 \int_{\p \Omega} |e^{\frac{w_{2,h}}{2}}- e^{\frac{w_{1,h}}{2}}| |\p_\nu u_2| |u_1|dS.
\end{equation}

For the first term in the right hand side of \eqref{b_terms_29}, using \eqref{b_terms_12}, \eqref{b_terms_13}, and \eqref{b_terms_2_rem}, we get 
\begin{equation}
\label{b_terms_30}
\begin{aligned}
 \int_{\p \Omega} |\p_\nu (e^{\frac{w_{2,h}}{2}}-& e^{\frac{w_{1,h}}{2}})| |u_2| |u_1|dS\le 
 \mathcal{O}(1)\|\p_\nu (e^{\frac{w_{2,h}}{2}}- e^{\frac{w_{1,h}}{2}})\|_{L^2(\p \Omega)}\|a_1a_2\|_{L^\infty(\p \Omega)}\\
 &+ \mathcal{O}(1)\|\p_\nu (e^{\frac{w_{2,h}}{2}}- e^{\frac{w_{1,h}}{2}})\|_{L^\infty(\p \Omega)}\big(\|a_1\|_{L^\infty(\p \Omega)}\|r_2\|_{L^2(\p \Omega)}\\
 &+ \|a_2\|_{L^\infty(\p \Omega)}\|r_1\|_{L^2(\p \Omega)}+ \|r_1\|_{L^2(\p \Omega)}\|r_2\|_{L^2(\p \Omega)}\big)=o(1), \quad h\to 0.
\end{aligned}
\end{equation}

To estimate the second term in the right hand side of \eqref{b_terms_29}, using 
\eqref{b_terms_13_1_1}, we see that 
\begin{equation}
\label{b_terms_31}
\begin{aligned}
 \int_{\p \Omega} |e^{\frac{w_{2,h}}{2}}- e^{\frac{w_{1,h}}{2}}| |\p_\nu u_2| |u_1|dS\le & \mathcal{O}(h^{-1})\int_{\p \Omega} |w_{2,h}-w_{1,h}| |v_2| |a_1+r_1|dS\\
 &+\mathcal{O}(1) \int_{\p \Omega} |w_{2,h}-w_{1,h}| |\p_\nu v_2| |a_1+r_1|dS.
\end{aligned}
\end{equation}
Using \eqref{eq_cond_Lipschitz}, \eqref{b_terms_18} and \eqref{b_terms_2_rem}, we get
\begin{equation}
\label{b_terms_32}
\begin{aligned}
 \int_{\p \Omega} |w_{2,h}-w_{1,h}| &|\p_\nu v_2| |a_1+r_1|dS\\
& \le \mathcal{O}(1)(\|w_{2,h}-w_{2}\|_{L^\infty(\p \Omega)}+ \|w_{1,h}-w_{1}\|_{L^\infty(\p \Omega)} )\|\p_\nu v_2\|_{L^2(\p \Omega)}\\
 &(\|a_1\|_{L^\infty(\p \Omega)}+\|r_1\|_{L^2(\p\Omega)})=o(1), \quad h\to 0.
\end{aligned}
\end{equation}

Using \eqref{eq_cond_Lipschitz} \eqref{eq_cond_delta_1}, and \eqref{b_terms_2_rem}, we obtain that 
\begin{equation}
\label{b_terms_33_old}
\begin{aligned}
\mathcal{O}&(h^{-1})\int_{\p \Omega} |w_{2,h}-w_{1,h}| |v_2| |a_1+r_1|dS\le \mathcal{O}(h^{-1}) \| w_{2,h}-w_{1,h}\|_{L^2(\p \Omega)}\|a_1a_2\|_{L^\infty(\p \Omega)}\\
&+ \mathcal{O}(h^{-1}) \| w_{2,h}-w_{1,h}\|_{L^\infty(\p \Omega)}( \|a_1\|_{L^\infty(\p \Omega)}\|r_2\|_{L^2(\p \Omega)} + \|a_2\|_{L^\infty(\p \Omega)}\|r_1\|_{L^2(\p \Omega)}\\
 &+ \|r_1\|_{L^2(\p \Omega)}\|r_2\|_{L^2(\p \Omega)} )=o(1), \quad h\to 0.
\end{aligned}
\end{equation}

It follows from \eqref{b_terms_31}, \eqref{b_terms_32} and \eqref{b_terms_33_old} that 
\begin{equation}
\label{b_terms_33}
 \int_{\p \Omega} |e^{\frac{w_{2,h}}{2}}- e^{\frac{w_{1,h}}{2}}| |\p_\nu u_2| |u_1|dS=o(1), \quad h\to 0.
\end{equation}

We conclude from \eqref{b_terms_29}, \eqref{b_terms_30} and \eqref{b_terms_33} that \eqref{b_terms_28_1} holds. In view of \eqref{b_terms_28} and \eqref{b_terms_def_J_sum}, we have therefore established \eqref{eq_2_11_2}.  The proof is complete.
\end{proof}

\subsection{Recovery of conductivity}

We conclude from Lemma \ref{lem_integral_identity}, Lemma \ref{lem_full_data_lem}, and Lemma \ref{lem_boundary_terms_int} that 
\begin{equation}
\label{eq_recovery_1}
\int_{\Omega}\big[- \nabla \gamma_1^{1/2}\cdot \nabla (\gamma_1^{-1/2} a_1a_2)+\nabla \gamma_2^{1/2}\cdot \nabla (\gamma_2^{-1/2}a_1a_2) \big]dx=0,
\end{equation}
for any $a_j\in C^\infty(\overline{\hat \Omega})$ such that 
\begin{equation}
\label{eq_recovery_2}
(\nabla \varphi+i\nabla \psi)\cdot \nabla a_j+\frac{1}{2}(\Delta \varphi+i\Delta \psi)a_j=0\quad \text{in}\quad \hat \Omega, 
\end{equation}
$j=1,2$.  Recall that 
\[
q_j=- \nabla \gamma_j^{1/2}\cdot \nabla \gamma_j^{-1/2} +\frac{1}{2}\Delta \log\gamma_j\in H^{-1}(\R^n)\cap \mathcal{E}'(\R^n).
\]
Letting $q=q_1-q_2$, and using the fact that $\gamma_1=\gamma_2$ on $\R^n\setminus\overline{\Omega}$, we conclude that $\supp(q)\subset \overline{\Omega}$.   

Let $\chi\in C_0^\infty(\hat \Omega)$ and $\chi=1$ near $\overline{\Omega}$. When $\phi\in C^\infty(\hat \Omega)$, we have
\begin{align*}
q(\phi)=&q(\chi \phi)=\int_\Omega (-\nabla \gamma_1^{1/2}\cdot \nabla \gamma_1^{-1/2}+\nabla \gamma_2^{1/2}\cdot \nabla \gamma_2^{-1/2})\phi dx\\
&-\frac{1}{2} \int_\Omega (\nabla \log \gamma_1-\nabla \log\gamma_2)\cdot \nabla \phi dx, 
\end{align*}
and therefore,  \eqref{eq_recovery_1} implies that 
\begin{equation}
\label{eq_recovery_3}
q(a_1a_2)=0.
\end{equation}

Recall that the functions $\varphi$ and $\psi$ in \eqref{eq_recovery_2} are defined by \eqref{eq_varphi} and \eqref{eq_psi}, respectively, using the fixed point $x_0\in \R^n\setminus\overline{\text{ch}(\Omega)}$. We shall denote them by $\varphi=\varphi_{x_0}$ and $\psi=\psi_{x_0}$ to emphasize the dependence on $x_0$. 

Now by the assumptions of Theorem \ref{thm_main},  
\[
\Lambda_{\gamma_1}f|_{\tilde F}=\Lambda_{\gamma_2}f|_{\tilde F}, \quad f\in H^{\frac{1}{2}}(\p \Omega), 
\] 
where $\tilde F$ is an open neighborhood of the front face $F(x_0)$, defined in \eqref{eq_int_front}. 
Therefore, there exists a neighborhood of $x_0$, $\text{neigh}(x_0,\R^n)$, such that for any $\tilde x_0\in \text{neigh}(x_0,\R^n)$, we have $F(\tilde x_0)\subset \tilde  F$. Here $F(\tilde x_0)$ is the front face with respect to $\tilde x_0$.  

Associated to $\tilde x_0$, we have the functions $\varphi_{\tilde x_0}$ and $\psi_{\tilde x_0}$, defined as in \eqref{eq_varphi} and \eqref{eq_psi}, and we have  
\[
\varphi_{\tilde x_0}(x)=\varphi_{x_0}(x-y), \quad \psi_{\tilde x_0}(x)=\psi_{x_0}(x-y),
\]
where $y=\tilde x_0-x_0\in \text{neigh}(0,\R^n)$. Observe that  $\varphi_{\tilde x_0},\psi_{\tilde x_0}\in C^\infty(\overline{\hat \Omega})$ for all $\tilde x_0\in \text{neigh}(x_0, \R^n)$. 

The analog of \eqref{eq_recovery_2}  with $\varphi_{x_0}$, $\psi_{x_0}$ replaced by $\varphi_{\tilde x_0}$, $\psi_{\tilde x_0}$ is solved by $a_j(\cdot-y)$, and thus, \eqref{eq_recovery_3} is valid for the translated distribution,  
\begin{equation}
\label{eq_recovery_4}
q(a_1(\cdot-y)a_2(\cdot-y))=0,
\end{equation}
for all $y\in  \text{neigh}(0, \R^n)$. 

Now let $\Psi_\tau$ be the usual mollifier, defined by \eqref{eq_mollifier_def} with a radial function $\Psi$. Then $q*\Psi_\tau\in C^\infty_0(\R^n)$, and for $\tau$ small, we have 
\[
\supp(q*\Psi_\tau)\subset \subset \{\chi=1\}^0\subset \subset \hat \Omega.
\]
It follows from 
\cite[Theorem 4.1.4]{Hormander_book_1} that for $\tau$ small,
\[
(q*\Psi_\tau)(a_1a_2)= (q*\Psi_\tau)(\chi a_1a_2)=q((\chi a_1a_2)* \Psi_\tau).
\]

Using that 
\[
(\chi a_1a_2)* \Psi_\tau=\lim_{\eta\to 0} \eta^n\sum_{k\in \Z^n}(\chi a_1a_2)(x-k\eta) \Psi_\tau(k\eta),
\]
where the convergence is uniform with all derivatives as $\eta\to 0$, see \cite[Lemma 4.1.3]{Hormander_book_1}, together with  \eqref{eq_recovery_4}, we get 
\[
q((\chi a_1a_2)* \Psi_\tau)=\lim_{\eta\to 0} \eta^n\sum_{k\in \Z^n} q((\chi a_1a_2)(\cdot -k\eta)) \Psi_\tau(k\eta)=0.
\]
Therefore, 
\[
(q*\Psi_\tau)(a_1a_2)=0,
\]
for all $\tau>0$ small.  Here $q*\Psi_\tau$ is smooth, and thus, we can apply the analysis of \cite[Section 6]{DKSU_2007} exactly as it stands, which allows us to conclude that $q*\Psi_\tau=0$.   Letting $\tau\to 0$, we get $q=0$, since $q*\Psi_\tau\to q$ in $\mathcal{E}'(\R^n)$. 

Finally,  by  \cite[Lemma 5.2]{Hab_Tataru}, we conclude that $\gamma_1=\gamma_2$ in $\R^n$. The proof of Theorem \ref{thm_main} is complete.

\begin{appendix}

\section{Approximation estimates}

\label{app_estimates}

The purpose of this appendix is to collect some approximation results which are used repeatedly in the main part of the paper. The estimates are well known and are given here for the convenience of the reader, see  \cite{Zhang_Guo_2012}.

In what follows, let $\Psi_\tau(x)=\tau^{-n}\Psi(x/\tau)$, $\tau>0$, be the usual mollifier with $\Psi\in C^\infty_0(\R^n)$, $0\le \Psi\le 1$, and 
$\int \Psi dx=1$.

\begin{lem}[{\cite[Lemma 2.1]{Zhang_Guo_2012}}]
\label{lem_Zhang_est}
Let $b\in H^{1/2}(\R^n)$.  Then $b_\tau=b*\Psi_\tau\in (C^\infty\cap H^{1/2})(\R^n)$,
\begin{equation}
\label{eq_2_10_7}
\|b-b_\tau\|_{L^2(\R^n)}=o(\tau^{1/2}), \quad \tau\to 0,
\end{equation} 
and
\begin{equation}
\label{eq_2_10_8}
\|b_\tau\|_{L^2(\R^n)}=\mathcal{O}(1), \quad  \|\nabla b_\tau\|_{L^2(\R^n)}=o(\tau^{-1/2}), \quad \tau\to 0.
\end{equation}
\end{lem}

\begin{proof}
We have
\begin{equation}
\label{eq_2_10_9}
\int_{\R^n} (1+|\xi|^2)^{1/2}|\hat b(\xi)|^2d\xi <\infty.
\end{equation}
Using that $\hat \Psi_\tau(\xi)=\hat \Psi(\tau \xi)$, we get $\hat b_\tau(\xi)=\hat b(\xi)\hat \Psi(\tau\xi)$, and therefore, 
\begin{equation}
\label{eq_2_10_10}
\begin{aligned}
\frac{1}{\tau}\|b-b_\tau\|^2_{L^2(\R^n)}=\frac{(2\pi)^{-n}}{\tau}\int_{\R^n}|1-\hat \Psi(\tau\xi)|^2 |\hat b(\xi)|^2d\xi=\int_{\R^n} g(\tau\xi) |\xi||\hat b(\xi)|^2d\xi, 
\end{aligned}
\end{equation}
where 
\[
g(\eta):=(2\pi)^{-n}\frac{|1-\hat \Psi(\eta)|^2}{|\eta|}.
\]
As $\hat \Psi(0)=1$, we have $g(0)=0$, and furthermore, since $\Psi\in C^\infty_0(\R^n)$, we conclude that $g$ is continuous and bounded.  By Lebesgue's dominated convergence theorem, applied to  \eqref{eq_2_10_10}, in view of \eqref{eq_2_10_9}, we get 
that  $\frac{1}{\tau}\|b-b_\tau\|^2_{L^2(\R^n)}\to 0$ as $\tau\to 0$, proving \eqref{eq_2_10_7}. 

The first part of \eqref{eq_2_10_8} is clear and to see the second part, we write 
\begin{align*}
\tau \|\p_{x_j} b_\tau\|^2_{L^2(\R^n)}=(2\pi)^{-n}\tau \int_{\R^n}|\xi_j|^2|\hat \Psi(\tau\xi)|^2|\hat b(\xi)|^2dx\le \int_{\R^n} \tilde g(\tau\xi)|\xi| |\hat b(\xi)|^2d\xi,
\end{align*}
where 
\[
\tilde g(\eta):=(2\pi)^{-n}|\eta||\hat \Psi(\eta)|^2
\]
is continuous and bounded and $\tilde g(0)=0$. By Lebesgue's dominated convergence theorem, we conclude that $\tau \|\p_{x_j} b_\tau\|^2_{L^2(\R^n)}=o(1)$, as $\tau\to 0$. The proof is complete. 
\end{proof}

\begin{lem}
\label{lem_app_delta}
Assume that $w\in C^1_0(\R^n)$ or $w\in H^{\frac{3}{2}+\delta}(\R^n)$, $\delta>0$ fixed.  Let  $w_h=w*\Psi_h$, where $\Psi_h$ is defined using a radial function $\Psi$. Let $\Omega\subset\R^n$, $n\ge 2$, be a bounded open set with $C^2$ boundary. 
Then 
\begin{equation}
\label{eq_cond_delta_1}
\|w_h-w\|_{L^2(\p \Omega)}=o(h),
\end{equation}
\begin{equation}
\label{eq_cond_delta_2}
\|\nabla w_h-\nabla w\|_{L^2(\p \Omega)}=o(1),
\end{equation}
as $h\to 0$.
\end{lem}

\begin{proof}

Let first $w\in C^1_0(\R^n)$.  To prove \eqref{eq_cond_delta_1} in this case, we shall show that 
\begin{equation}
\label{app_-1}
\|w_h-w\|_{L^\infty(\R^n)}=o(h), \quad h\to 0. 
\end{equation}
By the fundamental theorem of calculus, we get
\begin{equation}
\label{app_0}
\begin{aligned}
h^{-1}(w_h(x)-w(x))&=h^{-1}\int_{\R^n} (w(x-hy)-w(x))\Psi(y)dy\\
&=h^{-1} \int_{\R^n}\bigg(\int_0^1\frac{d}{dt} w(x-thy)dt\bigg) \Psi(y)dy\\
&= h^{-1}\int_{\R^n}\int_0^1 \nabla w(x-thy) \cdot (-hy)\Psi(y)dtdy.
\end{aligned}
\end{equation}
Using that $\Psi$ is even, we have
\begin{equation}
\label{app_1}
\int_{\R^n} (X\cdot y)\Psi(y)dy=0, \quad X\in \R^n.
\end{equation}
It follows from \eqref{app_0} and \eqref{app_1} that uniformly in $x$,
\[
h^{-1}(w_h(x)-w(x))=-\int_{\R^n}\int_0^1 (\nabla w(x-thy)-\nabla w(x)) \cdot ydt \Psi(y)dy=o(1),
\]
as $h\to 0$, which shows \eqref{app_-1}.   Here we have used that $\nabla w$ is uniformly continuous. 

In the case of $w\in C^1_0(\R^n)$, \eqref{eq_cond_delta_2} follows from the uniform continuity of $\nabla w$. 

Let now $w\in H^{\frac{3}{2}+\delta}(\R^n)$, $\delta>0$, and let us show \eqref{eq_cond_delta_1}. First by the trace theorem, we have 
\begin{equation}
\label{app_1_2}
\|w_h-w\|_{L^2(\p \Omega)}\le C\|w_h-w\|_{H^{\frac{1}{2}+\delta}(\R^n)}.
\end{equation}
We write
\begin{equation}
\label{app_2}
\begin{aligned}
h^{-2} \|w_h-w\|^2_{H^{\frac{1}{2}+\delta}(\R^n)}&=(2\pi)^{-n}h^{-2}\int_{\R^n} (1+|\xi|^2)^{\frac{1}{2}+\delta}|1-\hat\Psi(h\xi)|^2|\hat w(\xi)|^2d\xi\\
&\le \int_{\R^n}g(h\xi)(1+|\xi|^2)^{\frac{3}{2}+\delta} |\hat w(\xi)|^2d\xi,
\end{aligned}
\end{equation}
where 
\[
g(\eta)=(2\pi)^{-n}\frac{|1-\hat \Psi(\eta)|^2}{|\eta|^2}.
\]
Since $\Psi$ is radial, have $\nabla \hat \Psi(0)=-i\int_{\R^n}x\Psi (x)dx=0$.  Using this together with the fact that $\hat \Psi(0)=1$, and that $\hat \Psi\in \mathcal{S}(\R^n)$, we conclude that $g$ is continuous and bounded with $g(0)=0$. Hence, Lebesgue's dominated convergence theorem, applied to \eqref{app_2}, gives that $ \|w_h-w\|_{H^{\frac{1}{2}+\delta}(\R^n)}=o(h)$ as $h\to 0$, and therefore, in view of \eqref{app_1_2}, we see \eqref{eq_cond_delta_1}.

Similarly, using that $w\in H^{\frac{3}{2}+\delta}(\R^n)$ and  Lebesgue's dominated convergence theorem,  we get 
\[
\|\nabla w_h-\nabla w\|^2_{H^{\frac{1}{2}+\delta}(\R^n)}\le (2\pi)^{-n}\int_{\R^n} (1+|\xi|^2)^{\frac{3}{2}+\delta}|1-\hat \Psi(h\xi)|^2 |\hat w(\xi)|^2d\xi=o(1),
\]
as $h\to 0$. Hence, by the trace theorem, we get \eqref{eq_cond_delta_2}. The proof is complete. 
\end{proof}

We shall also need the following obvious estimates.
\begin{lem}
Let $w\in (W^{1,\infty}\mathcal \cap \mathcal{E}')(\R^n)$. Then 
\begin{equation}
\label{eq_cond_Lipschitz}
\sup_{\R^n} | w_h- w|\le \mathcal{O}(h),
\end{equation}
\begin{equation}
\label{eq_cond_Lipschitz_nabla}
\sup_{\R^n}| w_h|\le \mathcal{O}(1),\quad \sup_{\R^n}| \nabla w_h|\le \mathcal{O}(1),
\end{equation}
as $h\to 0$.
\end{lem}

\end{appendix}

\section*{Acknowledgements}
K.K  is grateful to Russell Brown for some very helpful discussions. We would like to thank Mikko Salo for bringing the work  \cite{Rodriguez} to our attention. 
The research of K.K. is partially supported by the National Science Foundation (DMS 1500703). The research of G.U. is partially supported by the National Science Foundation, Simons Fellowship, and the Academy of Finland.

\end{document}